\documentclass[print,wide]{template/draft}

\graphicspath{{images/}}
\usepackage{multirow}

\begin{document}
\title{A Geometric Inverse Source Problem for Stochastic Parabolic Equations}

\author{Yunzhang Li\thanks{Research Institute of Intelligent Complex Systems, Fudan University, Shanghai 200433, China. {\small\it E-mail:} {\small\tt li\_yunzhang@fudan.edu.cn.} This author is supported by the National Natural Science Foundation of China (12301566), the Science and Technology Commission of Shanghai Municipality (23JC1400300), and the Pujiang Program of Shanghai Magnolia Talent Plan (24PJD002).}, \quad
Qi L\"u\thanks{School of Mathematics, Sichuan University, Chengdu, 610064, China. {\small\it E-mail:} {\small\tt lu@scu.edu.cn.} This author is supported by the National Natural Science Foundation of China under grant 12025105.},\quad
Meizhi Qian\thanks{Department of Mathematics, Lehrstuhl f\"ur Dynamics, Control, Machine Learning and Numerics (Alexander von Humboldt-Professur), Friedrich-Alexander-Universit\"at Erlangen--N\"urnberg, Cauerstr.~11, 91058, Erlangen, Germany. {\small\it E-mail:} {\small\tt qmz171215@163.com.}},\quad
Yu Wang\thanks{School of Mathematics, Southwest Jiaotong University, Chengdu, 611756, China. {\small\it E-mail:} {\small\tt yuwangmath@163.com.} This author is supported by the National Natural Science Foundation of China under grant
12401589 and by Sichuan Science and Technology Program under grant
2026NSFSC0777.}
}

\date{}

\maketitle

\vspace{-3mm}
 
\begin{abstract}
This paper addresses the geometric inverse problem of simultaneously recovering two unknown deterministic source supports in a stochastic parabolic equation, where one source appears in the drift term and the other in the diffusion term. We establish that partial boundary flux measurements alone uniquely determine both supports. Moreover, we prove the existence of minimizers for a perimeter-regularized objective functional. For smooth interfaces, we derive a Hadamard-type boundary representation of the shape derivative. To the best of our knowledge, this is the first such formula for the simultaneous recovery of a drift-source support and a diffusion-source support in an SPDE setting. The derivative exhibits a genuinely stochastic two-channel structure: the adjoint state governs the sensitivity of the drift-source interface, while the martingale component controls the sensitivity of the diffusion-source interface. Based on this formula, we develop a shape-gradient reconstruction method. Numerical experiments demonstrate its effectiveness and its capacity to distinguish between the two source channels under both full and partial boundary observations.
\end{abstract}


\vspace{-3mm}

\section{Introduction}
\label{sec:introduction}

Geometric inverse problems seek to recover an unknown interface, inclusion, or source support from indirect measurements generated by a partial differential equation. When the unknown is a geometric set represented by its characteristic function, shape calculus provides a natural variational framework (see     
\cite{Sokolowski1992,Delfour2011,Haslinger2003,Maggi2012,Ambrosio2000}
and references therein). 
This approach is well established for deterministic equations. In the present paper, we extend it to stochastic partial differential equations and study the simultaneous recovery of two deterministic source supports in a stochastic parabolic equation from partial boundary flux measurements. One support enters the drift channel, and the other enters the diffusion channel.

Throughout, $T>0$ is fixed, and $G\subset\mathbb{R}^n$, $n\geq 2$, is a
bounded $C^2$ domain. The observation set $\Gamma$ is a nonempty relatively
open subset of $\partial G$, and $K\subset\subset G$ is a fixed compact set.
We work on a complete filtered probability space
$(\Omega,\mathcal{F},\mathbf{F},\mathbb{P})$ carrying a one-dimensional
standard Brownian motion $\{W(t)\}_{t \geq 0}$. The filtration
$\mathbf{F}=\{\mathcal{F}_{t}\}_{t \geq 0}$ is the augmented natural
filtration. We denote by $\mathbb{F}$ the progressive $\sigma$-field with
respect to $\mathbf{F}$ and by $\nu$ the unit outward normal on $\partial G$.

Let $\mathcal H$ be a Banach space. We use the following standard spaces of
$\mathcal H$-valued stochastic processes:
\begin{itemize}
\item $L^{2}_{\mathbb{F}}(0,T;\mathcal H)$ denotes the space of all
$\mathcal H$-valued $\mathbf F$-adapted processes $X(\cdot)$ with
$\mathbb E\bigl(\|X(\cdot)\|_{L^2(0,T;\mathcal H)}^2\bigr)<\infty$.
\item $L^{\infty}_{\mathbb{F}}(0,T;\mathcal H)$ denotes the space of all
essentially bounded $\mathcal H$-valued $\mathbf F$-adapted processes.
\item $L^{2}_{\mathbb{F}}(\Omega;C([0,T];\mathcal H))$ denotes the space of
all $\mathcal H$-valued $\mathbf F$-adapted continuous processes
$X(\cdot)$ satisfying
$\mathbb E\bigl(\|X(\cdot)\|_{C([0,T];\mathcal H)}^2\bigr)<\infty$.
\end{itemize}
All these spaces are equipped with their usual norms; see
\cite[Section~2.6]{Lue2021a}.

We consider the stochastic parabolic equation with two geometric source terms:\vspace{-2mm}
\begin{align}
\label{eqStateGIP}
\begin{cases}
d u=\left(\Delta u+\chi_{D_0}+a u\right)\,dt
+ \left(b u+\chi_{D_1}\right)\,dW(t) & \text{in }(0, T) \times G, \\
u=0, & \text{on }(0, T) \times \partial G, \\
u(0, x)=u_0 & \text{in } G,
\end{cases} 
\end{align} 
where $a \in L^{\infty}_{\mathbb{F}}(0,T; L^{\infty}(G))$,
$b \in L^{\infty}_{\mathbb{F}}(0,T; W^{1,\infty}(G))$,
$u_0\in H^1_0(G)$, and $\chi_{D_i}$ is the
characteristic function of a measurable set $D_i\subset K$, $i=0,1$.
The set $D_0$ supports the drift source and $D_1$ the diffusion
source; the two sets may intersect, coincide, or one may be empty.

\begin{problem}[Geometric inverse source problem]
\label{probGIP}
Determine the source supports $(D_0,D_1)$ from the
partial boundary flux $\partial u/\partial \nu$ on $(0,T)\times\Gamma$.
\end{problem}

\begin{remark}
A concrete motivation for \cref{probGIP} is thermal source localization under
uncertain operating conditions. In electronic packages, inverse parabolic-conduction
methods locate internal hot spots from surface measurements
\cite{KraneEtAl2022}, and stochastic models account for non-stationary random
thermal loads \cite{LiSu2022}. In a linearized
stochastic model,
$D_0$ represents the region of persistent mean parabolic release, while $D_1$
localizes the amplitude of random parabolic-release fluctuations.
\end{remark}

Geometric reconstruction of discontinuous sources and inclusions has been extensively studied for deterministic parabolic equations. Hettlich and Rundell investigated the recovery of source supports and discontinuous sources in parabolic equations \cite{Hettlich2001}. Their work treats a discontinuous source in the parabolic equation and can be seen as a direct deterministic precursor to our model. Shape-optimization approaches to parabolic interface identification were developed in \cite{HenrotSokolowski1998,HarbrechtTausch2011,HarbrechtTausch2013,BruggerHarbrechtTausch2021}, boundary-data formulations in \cite{ChapkoKressYoon1998}, and sparse boundary-flux reconstruction in \cite{LinZhangZhang2022,LinOuZhangZhang2024}.  Recent advances include reduced-order reconstruction of geometric parabolic sources \cite{HuZhangZhu2025} and moving shape identification for parabolic sources and potentials \cite{FanZhu2025}. These results confirm the effectiveness of geometry-aware formulations for unknowns with sharp interfaces, but all concern deterministic evolution equations.

A distinct direction is shape optimization under uncertainty, where deterministic geometry is paired with random loads or coefficients; see \cite{ContiEtAl2009,BurtscheidtEtAl2023}. The closest framework is the stochastic product-shape optimization of \cite{GeiersbachLoayzaRomeroWelker2021,GeiersbachLoayzaRomeroWelker2023}: the unknown $(D_0,D_1)$ is a deterministic pair of shapes, the state is random, and the objective is an expected shape functional. The mathematical structures, however, differ in three key respects. First, their forward models rely on random-parameter PDEs that reduce to deterministic elliptic problems per realization, whereas ours is a genuinely stochastic parabolic equation. Second, their shapes are nonintersecting, while our $D_0$ and $D_1$ are additive source supports that may overlap. Third, their shape gradient uses a deterministic adjoint per realization, whereas the sensitivity of our diffusion support $D_1$ is governed by the martingale component $q$ of a backward stochastic adjoint equation.

Inverse problems governed by genuinely  stochastic PDEs are far less developed. As emphasized in the survey \cite{LuZhang2024}, the literature remains limited and many fundamental questions are open. L\"u \cite{Lu2012} proved a single-source uniqueness result for a stochastic parabolic equation via a global Carleman estimate. Simultaneous identification of function-valued sources in both the drift and diffusion channels has been obtained for stochastic wave equations \cite{Yuan2015}, stochastic parabolic equations \cite{Yuan2021}, and stochastic Grushin equations \cite{Yan2022}.   More recently, Wang, Yang, and Zhang \cite{WangYangZhang2026} showed that partial boundary-flux data determine the temporal strength $|f(t)|$ of a separable random source $g(x)f(t)d W(t)$ in stochastic parabolic and wave equations. A distinct line of research concerns frequency-domain inverse scattering with random-field sources; for instance, \cite{LiLiWang2022} recovers the principal symbols of the covariance and relation operators of a generalized Gaussian random source from far-field patterns, rather than identifying sources in the drift and diffusion terms.

Moreover, available inverse source results for SPDEs rely on structural assumptions. For instance, the separated-source framework in \cite[Chapters~3--4]{Lue2024} requires a prescribed nonvanishing factor together with geometric, regularity, separability, and sometimes terminal-measurement conditions. These hypotheses do not accommodate a general characteristic function $\chi_D$ of an unknown support. In our setting, the unknowns are the characteristic functions of two independently varying sets that may overlap, and we use only partial boundary-flux data without terminal observation. Differentiating a characteristic source yields a distribution supported on its interface, not an admissible $L^2$ or $H^1$ source, so existing Carleman arguments may not apply directly. Our uniqueness proof instead separates the two channels: expectation and a deterministic inverse-source argument for the drift source, and a stochastic transposition identity, It\^o isometry, and approximate boundary reachability for the diffusion source.

We formulate the reconstruction as a perimeter-regularized
shape optimization problem. Given noisy partial boundary data
$g\in L^2_{\mathbb F}(0,T;L^2(\Gamma))$, we seek\vspace{-2mm}
\begin{align}
\label{eqOptGIP}
\min_{(D_0,D_1)\in\mathcal A_K} \mathcal{J}(D_0,D_1),
\end{align}
where\vspace{-5mm}
$$
\mathcal{J}(D_0,D_1) := \frac{1}{2} \mathbb{E} \int_0^T \int_{\Gamma}\left(\frac{\partial u}{\partial \nu}-g\right)^2 \,d\sigma\,dt
+
\beta_0 \mathcal{P}_G(D_0)+\beta_1 \mathcal{P}_G(D_1),\vspace{-1mm}
$$
with regularization parameters $\beta_0,\beta_1>0$ and admissible class\vspace{-2mm}
\begin{equation}
\label{eqAdmissibleClassK}
\mathcal A_K
:=
\{(D_0,D_1): D_0,D_1\text{ are measurable subsets of }K\}.\vspace{-2mm}
\end{equation}
The relative perimeter $\mathcal P_G(D)$ of a measurable set $D\subset G$ is\vspace{-2mm}
\begin{equation}
\label{eqRelativePerimeter}
\mathcal P_G(D)
:=
\sup\Big\{
\int_D \operatorname{div}\phi\,dx
:\,
\phi\in C_c^1(G;\mathbb{R}^n),\
\|\phi\|_{L^\infty(G)}\leq 1
\Big\}.\vspace{-2mm}
\end{equation}
The boundary flux $\partial u/\partial \nu|_\Gamma$ is well defined because the sources are compactly supported in the interior of $G$; see \cref{lem:forward-local-reg}.

Beyond uniqueness, the present formulation provides a perimeter-regularized existence theory, a shape-gradient reconstruction method, and a genuinely stochastic boundary shape derivative. The central novelty is the interplay between the two source channels and the stochastic adjoint structure. In the shape-optimization framework, the natural adjoint is a backward stochastic parabolic equation whose solution consists of two adapted components: the adjoint state $p$ and the martingale component $q$. The shape derivative, obtained in boundary form in the spirit of deterministic boundary shape-gradient methods \cite{HiptmairPaganiniSargheini2015,Zhu2019}, exhibits a genuinely stochastic two-channel structure: $p$ determines the sensitivity of the drift interface $\partial D_0$, while $q$ determines the sensitivity of the diffusion interface $\partial D_1$. This structure also creates a computational difficulty. Unlike a deterministic adjoint, the backward stochastic equation cannot be solved by direct time reversal; its backward recursion involves conditional expectations. One of these conditional expectations determines the martingale component, and its accurate approximation is a main computational challenge \cite{DunstProhl2016,Zhang2017,Li2022a,Lue2022}.

The main contributions are as follows.\vspace{-1.5mm}
\begin{enumerate}[(i)]
\item We formulate a two-support geometric inverse source problem for the
stochastic parabolic equation \cref{eqStateGIP} with partial boundary-flux data and
two independently varying source supports.\vspace{-1.5mm}

\item  We prove that partial boundary-flux data alone uniquely determine both source supports, including that of the diffusion source; see \cref{thm:uniqueness}. This result is of independent interest. Unlike the simultaneous identification framework of \cite{Yuan2021}, neither a terminal-time observation nor full boundary data is required: the flux is measured only on an arbitrary nonempty relatively open subset $\Gamma\subset\partial G$.\vspace{-1.5mm}

\item  We prove existence of minimizers for the perimeter-regularized
functional \cref{eqOptGIP} on the admissible class \cref{eqAdmissibleClassK}. The proof
combines compactness of finite-perimeter sets with a local hidden boundary
regularity estimate for interior $H^{-1}$ sources.\vspace{-1.5mm}

\item  We derive the boundary shape derivative formula
\cref{eq:shape-derivative-main}. The derivation requires a localization argument and
interior regularity near the moving interfaces, exploiting the compact support
of the sources to compensate for the limited regularity of stochastic parabolic
equations.\vspace{-1.5mm}

\item We propose a shape-gradient descent method based on the boundary
derivative formula \cref{eq:shape-derivative-main}. The method addresses the
backward stochastic adjoint bottleneck through finite element discretization and
least-squares Monte Carlo approximation of conditional expectations; numerical
experiments demonstrate effective reconstruction of both source supports under
full and partial boundary observations.
\end{enumerate}

The remainder of the paper is organized as follows.
\Cref{secForwardWellposedness} establishes forward well-posedness and boundary regularity.
\Cref{sec:uniqueness-reduction} proves uniqueness of the two source supports.
\Cref{sec:existence-minimizer} proves existence of minimizers for the perimeter-regularized objective.
\Cref{sec:shape-preliminaries} recalls shape-calculus preliminaries.
\Cref{secShapeSensitivityAnalysis} develops the adjoint equation, geometric-source linearization, interior regularity near moving interfaces, and the shape derivative formula.
\Cref{secNumericalAlgorithm,secNumericalExperiments} describe the numerical method and present the experiments.

\vspace{-2mm}

\section{Well-posedness and boundary regularity for forward equations}
\label{secForwardWellposedness}

This section records the well-posedness and boundary regularity results used in
the uniqueness, existence, and shape sensitivity arguments.

For every pair $(D_0,D_1)$ with $D_i\subset K$, the characteristic sources
belong to $L^2(G)$. By the classical well-posedness result for stochastic evolution
equations, \cref{eqStateGIP} admits a unique weak solution
$u \in L_{\mathbb{F}}^{2}\big(\Omega; C([0,T];L^2(G))\big)
\cap L_{\mathbb{F}}^{2}(0,T;H_0^1(G))$;
see, e.g., \cite[Theorem 3.24]{Lue2021a}.

For later use, consider the linear equation \vspace{-1.5mm}
\begin{equation}
\label{eqStateGeneral}
\begin{cases}
d y =(\Delta y+a y+ \Phi)\,dt + (b y + \Psi)\,dW(t) & \text{in }(0, T) \times G, \\
y=0 & \text{on }(0, T) \times \partial G, \\
y(0, x)=y_0(x) & \text{in } G.
\end{cases}
\end{equation}

\begin{definition}
\label{defStrongSolution}
A process $y$ is called a strong solution to \cref{eqStateGeneral} if $
y \in L_{\mathbb{F}}^{2}\big(\Omega ; L^{2}(0, T ; H^{2}(G) \cap H^{1}_0(G))\big)
\cap
L_{\mathbb{F}}^{2}\big(\Omega ; C([0,T]; H^{1}_0(G))\big)$, 
and if, for all $t \in [0,T]$,\vspace{-1.5mm}
\begin{align*}
y(t) - \int_0^t \big(\Delta y(s)+a(s)y(s)\big) \,ds
=
y_0 + \int_0^t \Phi(s) \,ds + \int_0^t (b(s) y(s) + \Psi(s)) \,dW(s),
\qquad \mathbb{P}\text{-a.s.}
\end{align*}
\end{definition}
\vspace{-2mm}
\begin{theorem}[{\cite[Theorem 7.3]{Neerven2012}}]
\label{thmWellPosedness}
Let $y_0 \in H^{1}_0(G)$, let $\Phi \in L^2_{\mathbb F}(0,T;L^2(G))$, and let $\Psi \in L^2_{\mathbb F}(0,T;H_{0}^1(G))$.
If $a\in L^{\infty}_{\mathbb F}(0,T;L^\infty(G))$ and
$b\in L^{\infty}_{\mathbb F}(0,T;W^{1,\infty}(G))$, then
equation \cref{eqStateGeneral} admits a unique strong solution $y$.
Furthermore, there exists a constant $C > 0$ such that \vspace{-1.5mm}
\begin{align*}
\|y\|_{L^{2}_{\mathbb{F}}(0, T; H^{2}(G) \cap H^{1}_0(G)) \cap L^{2}_{\mathbb{F}}(\Omega; C([0,T]; H^{1}_{0}(G)))}
\leq
C \big(
\|\Phi\|_{L^2_{\mathbb F}(0,T;L^2(G))}
+
\|\Psi\|_{L^2_{\mathbb F}(0,T;H_{0}^1(G))}
+
\|y_0\|_{H^{1}_0(G)}
\big).
\end{align*}
\end{theorem}

Here and in the sequel, $C$ denotes a generic positive constant depending only on $G$, $T$, and the coefficient bounds $\|a\|_{L^\infty_{\mathbb F}(0,T;L^\infty(G))}$ and $\|b\|_{L^\infty_{\mathbb F}(0,T;W^{1,\infty}(G))}$, and it may change from line to line. Constants denoted by $C_M$ or $C_K$ may also depend on the indicated interior support, in particular on its distance from $\partial G$.

For $f \in H^{-1}(G)$ and a closed set $K\subset G$, the notation $\operatorname{supp}f \subset K$ means that $\langle f ,\varphi\rangle_{H^{-1}(G),H_0^1(G)}=0$ for all $\varphi\in C_c^\infty(G\setminus K)$. For a stochastic process, this condition is imposed for a.e. $(t,\omega)$.

For $f_0,f_1\in L^{2}_{\mathbb{F}}(0,T;H^{-1}(G))$ and
$z_{0}\in H^{-1}(G)$, consider the forward equation\vspace{-1.5mm}
\begin{equation}
\label{eq:forward-hminus1}
\begin{cases}
d z = (\Delta z + a z + f_0)\,dt + (b z + f_1)\,dW(t)
& \text{in } (0,T)\times G, \\
z = 0 & \text{on } (0,T)\times \partial G, \\
z(0) = z_0 & \text{in } G.
\end{cases}\vspace{-1.5mm}
\end{equation}
It follows from standard results for stochastic evolution equations (see, e.g.,~\cite[Theorem 3.24]{Lue2021a}) that \cref{eq:forward-hminus1} admits a unique mild solution
$z\in L^2_{\mathbb F}(\Omega;C([0,T];H^{-1}(G)))
\cap L^2_{\mathbb F}(0,T;L^2(G))$ and \vspace{-1.5mm}
\begin{equation} 
\label{eq:var-solution-bound}
\|z\|_{L^2_{\mathbb F}(\Omega;C([0,T];H^{-1}(G)))}\!
+\!
\|z\|_{L^2_{\mathbb F}(0,T;L^2(G))}\!
\le\!
C\big(
\|f_0\|_{L^2_{\mathbb F}(0,T;H^{-1}(G))}\!
+ \|f_1\|_{L^2_{\mathbb F}(0,T;H^{-1}(G))}\!
+
\|z_0\|_{H^{-1}(G)}
\big).\vspace{-1mm}
\end{equation}

\begin{lemma}
\label{lem:forward-local-caccioppoli}
Let $M \subset\subset G$ and $U \subset\subset G\setminus\overline M$ be open sets. Let
$f_0,f_1\in L^2_{\mathbb F}(0,T;H^{-1}(G))$ satisfy
$\operatorname{supp} f_i(t,\omega)\subset \overline M$, $i=0,1$,
for a.e. $(t,\omega)$. Let $z_0\in H_0^1(G)$, and let $z$ be the mild solution of \cref{eq:forward-hminus1}. Then\vspace{-1mm}
\begin{equation}
\label{eq:forward-local-caccioppoli}
\|z\|_{L^2_{\mathbb F}(0,T;H^1(U))}
\leq
C_{M,U}\big(
\|z_0\|_{H_0^1(G)}
+
\|f_0\|_{L^2_{\mathbb F}(0,T;H^{-1}(G))}
+
\|f_1\|_{L^2_{\mathbb F}(0,T;H^{-1}(G))}
\big),\vspace{-1mm}
\end{equation}
where $C_{M,U}$ depends on $G,T$, the coefficient bounds, and the distance
between $U$ and $\overline M$, but not on $z_0,f_0,f_1$.
\end{lemma}

\begin{proof}
\textbf{Step 1.} In this step, we prove the estimate first for smooth adapted sources,
after choosing a cutoff away from the source support. 

Choose an open set
$M_2$ such that $\overline M \subset\subset M_2 \subset\subset G$ and
$U \subset\subset G\setminus\overline{M_2}$.
Choose $\eta\in C_c^\infty(G\setminus\overline{M_2})$ such that $\eta=1$ on
$U$. Assume that $f_0,f_1$ take values in $C_c^\infty(M_2)$. Then
$f_0\in L^2_{\mathbb F}(0,T;L^2(G))$ and
$f_1\in L^2_{\mathbb F}(0,T;H_0^1(G))$.
Therefore \cref{thmWellPosedness}, applied to \cref{eqStateGeneral} with
$\Phi=f_0$ and $\Psi=f_1$, gives a strong solution satisfying
$z\in L^2_{\mathbb F}(0,T;H^2(G)\cap H_0^1(G))
\cap L^2_{\mathbb F}(\Omega;C([0,T];H_0^1(G)))$.

Set
$w:=\eta z$.
Since $\eta$ is deterministic and independent of $t$, multiplying the strong
equation by $\eta$ gives the following identity in $L^2(G)$:\vspace{-4mm}
\begin{align*}
d w 
=
\big(\eta\Delta z+a w\big)\,dt
+
b w\,dW(t).
\end{align*}
Here $\eta f_0=\eta f_1=0$, because
$\operatorname{supp}\eta\cap\overline{M_2}=\emptyset$ and
$\operatorname{supp}f_i\subset M_2$. Equivalently, using
$\Delta(\eta z)=\eta\Delta z+2\nabla\eta\cdot\nabla z+(\Delta\eta)z$,
the localized equation becomes\vspace{-2mm}
\begin{equation}
\label{eq:localized-caccioppoli-equation}
d w
=
\big(
\Delta w
-2\nabla\eta\cdot\nabla z
-(\Delta\eta)z
+a w
\big)\,dt
+
b w\,dW(t).\vspace{-2mm}
\end{equation}
All terms in \cref{eq:localized-caccioppoli-equation} are $L^2(G)$-valued in
the present smooth-source setting: $w\in H^2(G)\cap H_0^1(G)$, while the
commutator terms are in $L^2(G)$ because $z\in H_0^1(G)$ and $\eta$ is smooth.

Applying It\^o's formula to the $L^2(G)$-valued semimartingale $w=\eta z$
and using the first form of the cutoff equation, we obtain\vspace{-2mm}
\begin{align*}
d\|w(t)\|_{L^2(G)}^2
&=
2(w(t),\eta\Delta z(t)+a(t)w(t))_{L^2(G)}\,dt
+
\|b(t)w(t)\|_{L^2(G)}^2\,dt
+
2(b(t)w(t),w(t))_{L^2(G)}\,dW(t)
\\
&=
2(\eta^2 z(t),\Delta z(t))_{L^2(G)}\,dt
+
2(a(t)w(t),w(t))_{L^2(G)}\,dt
+
\|b(t)w(t)\|_{L^2(G)}^2\,dt
\\
&\quad
+
2(b(t)w(t),w(t))_{L^2(G)}\,dW(t).
\end{align*}
The integration by parts below is justified since
$\eta^2z\in H_0^1(G)$ and $\Delta z\in L^2(G)$:\vspace{-2mm}
\begin{align*}
(\eta^2z,\Delta z)_{L^2(G)}
&=
-\int_G \nabla(\eta^2z)\cdot\nabla z\,dx =
-\int_G \eta^2|\nabla z|^2\,dx
-
2\int_G \eta z\nabla\eta\cdot\nabla z\,dx
\\
&=
-\|\nabla(\eta z)\|_{L^2(G)}^2
+
\||\nabla\eta|z\|_{L^2(G)}^2.
\end{align*}
Substituting this identity, and recalling that $w=\eta z$, gives\vspace{-2mm}
\begin{align*}
&d\|\eta z(t)\|_{L^2(G)}^2 
+
2\|\nabla(\eta z(t))\|_{L^2(G)}^2\,dt
\\
&
=
\Big[
2\||\nabla\eta|z(t)\|_{L^2(G)}^2
+
2(a(t)\eta z(t),\eta z(t))_{L^2(G)}
+
\|b(t)\eta z(t)\|_{L^2(G)}^2
\Big]\,dt 
+
2(b(t)\eta z(t),\eta z(t))_{L^2(G)}\,dW(t).
\end{align*}
Taking expectations, using the boundedness of $a$ and $b$, and dropping the
nonnegative terminal term give, for every $t\in[0,T]$,\vspace{-3mm}
\begin{align*}
\mathbb E\int_0^t\|\nabla(\eta z(s))\|_{L^2(G)}^2\,ds
\leq
C_\eta\bigg(
\|z_0\|_{L^2(G)}^2
+
\mathbb E\int_0^t\|z(s)\|_{L^2(G)}^2\,ds
\bigg).
\end{align*}
Since $\eta=1$ on $U$, this implies\vspace{-2mm}
$$
\|z\|_{L^2_{\mathbb F}(0,T;H^1(U))}^2
\leq
C_\eta\big(
\|z_0\|_{L^2(G)}^2
+
\|z\|_{L^2_{\mathbb F}(0,T;L^2(G))}^2
\big).\vspace{-2mm}
$$
Combining this with \cref{eq:var-solution-bound}, and using
$\|z_0\|_{H^{-1}(G)}+\|z_0\|_{L^2(G)}\leq C\|z_0\|_{H_0^1(G)}$, yields
\cref{eq:forward-local-caccioppoli} for smooth adapted sources.

\textbf{Step 2.} In this step, we pass from smooth sources to rough sources.

For general $f_0,f_1\in L^2_{\mathbb F}(0,T;H^{-1}(G))$ supported in
$\overline M$, choose adapted smooth approximations
$f_0^k,f_1^k\in L^2_{\mathbb F}(0,T;C_c^\infty(M_2))$ such that
$f_i^k\to f_i$ in $L^2_{\mathbb F}(0,T;H^{-1}(G))$, $i=0,1$.
This can be obtained by predictable simple-process approximation in time and
spatial mollification inside $M_2$. Let $z^k$ be the corresponding solutions.
By \cref{eq:var-solution-bound},
$z^k\to z$ in $L^2_{\mathbb F}(0,T;L^2(G))$ and in
$L^2_{\mathbb F}(\Omega;C([0,T];H^{-1}(G)))$.
The smooth-source estimate gives a uniform bound for $\{\eta z^k\}_{k=1}^\infty$ in
$L^2_{\mathbb F}(0,T;H^1(G))$. Passing to a subsequence if necessary,
$\{\eta z^k\}_{k=1}^\infty$ converges weakly in this space. Since $\eta z^k\to\eta z$ strongly
in $L^2_{\mathbb F}(0,T;L^2(G))$, the weak limit of $\{\eta z^k\}_{k=1}^\infty$ is $\eta z$. Lower
semicontinuity gives the same estimate for $z$, and the proof is complete.
\end{proof}

\begin{lemma}[Local hidden boundary regularity for interior sources]
\label{lem:forward-local-reg}
Let $M \subset\subset G$ be open, and let
$f_0,f_1\in L^2_{\mathbb F}(0,T;H^{-1}(G))$ satisfy
$\operatorname{supp}f_i(t,\omega)\subset\overline M$, $i=0,1$,
for a.e. $(t,\omega)$. Let $z_0\in H_0^1(G)$, and let $z$ be the solution of \cref{eq:forward-hminus1}.
Then   $\partial z/\partial \nu\in L^2_{\mathbb F}(0,T;L^2(\partial G))$ and\vspace{-2mm}
\begin{equation}
\label{eq:normal-trace-est}
\bigg\|\frac{\partial z}{\partial \nu}\bigg\|_{L^2_{\mathbb F}(0,T;L^2(\partial G))}
\le
C_M\big(
\|f_0\|_{L^2_{\mathbb F}(0,T;H^{-1}(G))}
+
\|f_1\|_{L^2_{\mathbb F}(0,T;H^{-1}(G))}
+
\|z_0\|_{H_0^1(G)}
\big),\vspace{-2mm}
\end{equation}
where $C_M$ is independent of $f_0$, $f_1$, and $z_0$.
\end{lemma}

\begin{proof}
Choose an open set $M_2$ such that
$\overline M \subset\subset M_2 \subset\subset G$.
Choose $\zeta\in C^\infty(\overline G)$ such that $\zeta=1$ in a collar
neighborhood $U_\partial$ of $\partial G$ and
$\operatorname{supp}\zeta\cap\overline{M_2}=\emptyset$.
Set $v:=\zeta z$.
Choose an open set $M_1$ such that
$\operatorname{supp}\nabla\zeta\cup\operatorname{supp}\Delta\zeta
\subset M_1 \subset\subset G\setminus\overline{M_2}$.
By \cref{lem:forward-local-caccioppoli}, applied with the present set $M$ and
$U=M_1$, we
have
\begin{align}
\label{eq:local-H1-away-source}
\|z\|_{L^2_{\mathbb F}(0,T;H^1(M_1))}
\leq
C_M\big(
\|z_0\|_{H_0^1(G)}
+
\|f_0\|_{L^2_{\mathbb F}(0,T;H^{-1}(G))}
+
\|f_1\|_{L^2_{\mathbb F}(0,T;H^{-1}(G))}
\big).
\end{align}
Consequently,
$F_\zeta:=-2\nabla\zeta\cdot\nabla z-(\Delta\zeta)z
\in L^2_{\mathbb F}(0,T;L^2(G))$, 
and\vspace{-2mm}
$$
\|F_\zeta\|_{L^2_{\mathbb F}(0,T;L^2(G))}
\leq
C_M\big(
\|z_0\|_{H_0^1(G)}
+
\|f_0\|_{L^2_{\mathbb F}(0,T;H^{-1}(G))}
+
\|f_1\|_{L^2_{\mathbb F}(0,T;H^{-1}(G))}
\big).\vspace{-2mm}
$$
Since $\zeta=0$ in a neighborhood of $\overline M$, we have
$\zeta f_0=\zeta f_1=0$ in $H^{-1}(G)$. Using  the local $H^1$ regularity above, $v=\zeta z$
satisfies\vspace{-2mm}
\begin{equation}
\label{eq:cutoff-local-forward}
dv=(\Delta v+a v+F_\zeta)\,dt+b v\,dW(t),
\qquad
v|_{\partial G}=0,
\qquad
v(0)=\zeta z_0. \vspace{-2mm}
\end{equation}
Moreover, $\|\zeta z_0\|_{H_0^1(G)}\le C_M\|z_0\|_{H_0^1(G)}$.
By \cref{thmWellPosedness}, equation \cref{eq:cutoff-local-forward} has a
strong solution with\vspace{-2mm}
\begin{equation}
\label{eq:v-H2-est}
\|v\|_{L^2_{\mathbb F}(0,T;H^2(G))}
\le
C_M\big(
\|z_0\|_{H_0^1(G)}
+
\|f_0\|_{L^2_{\mathbb F}(0,T;H^{-1}(G))}
+
\|f_1\|_{L^2_{\mathbb F}(0,T;H^{-1}(G))}
\big).\vspace{-2mm}
\end{equation}
Since $\zeta=1$ on $U_\partial$, we have $v=z$ in a neighborhood of
$\partial G$. Therefore the normal trace of $z$ is defined by
$\partial z/\partial \nu:=\partial v/\partial \nu$ on
$(0,T)\times\partial G$.
The trace theorem for $H^2(G)\cap H_0^1(G)$ gives\vspace{-2mm}
$$
\bigg\|\frac{\partial z}{\partial \nu}\bigg\|_{L^2_{\mathbb F}(0,T;L^2(\partial G))}
=
\bigg\|\frac{\partial v}{\partial \nu}\bigg\|_{L^2_{\mathbb F}(0,T;L^2(\partial G))}
\le
C\|v\|_{L^2_{\mathbb F}(0,T;H^2(G))}.\vspace{-2mm}
$$
Combining this estimate with \cref{eq:v-H2-est} gives
\cref{eq:normal-trace-est}.
\end{proof}

In particular, for the state equation \cref{eqStateGIP} with sources $(\chi_{D_0},\chi_{D_1})$ supported in $K$, \cref{lem:forward-local-reg} gives\vspace{-2mm}
\begin{align}
\label{eqBoundaryObservationWellDefined}
\dfrac{\partial u}{\partial \nu} \in L^2_{\mathbb F}(0,T;L^2(\partial G)),
\end{align}
so that $\partial u/\partial \nu|_\Gamma$ is well defined in
$L^2_{\mathbb F}(0,T;L^2(\Gamma))$.

\vspace{-2mm}

\section{Uniqueness of the source supports}
\label{sec:uniqueness-reduction}

This section establishes one of the main theoretical results of the
paper: both source supports are uniquely determined by the partial boundary
flux alone. 

\begin{assumption}
\label{assumption:uniqueness}
Assume $a(t,x,\omega)=a_0(t,\omega)+c(x)$, with
$a_0\in L^\infty_{\mathbb{F}}(0,T)$ and $c\in L^\infty(G)$, and
$b\in L^\infty_{\mathbb{F}}(0,T;W^{1,\infty}(G))$.
\end{assumption}

\begin{theorem}
\label{thm:uniqueness}
Assume \cref{assumption:uniqueness}.
Let $(D_0^1,D_1^1), (D_0^2,D_1^2) \in \mathcal{A}_{K}$, and let
$u^{1}, u^{2}$ be the corresponding solutions, belonging to
$L_{\mathbb{F}}^{2}\big(\Omega; C([0,T];L^2(G))\big) \cap L_{\mathbb{F}}^{2}(0,T;H_0^1(G))$, with the initial datum
$u_0\in H_0^1(G)$ and the same coefficients $a,b$. If
$\partial u^1/\partial \nu=\partial u^2/\partial \nu$ in
$L^2_{\mathbb{F}}(0,T;L^2(\Gamma))$,
then
$\chi_{D_0^1}=\chi_{D_0^2}$ and
$\chi_{D_1^1}=\chi_{D_1^2}$ a.e. in $G$.
\end{theorem}

\begin{remark}
The same argument gives a function-source version, showing that the uniqueness
conclusion is not restricted to characteristic sources. If the drift sources in
\cref{eqStateGIP} are replaced by deterministic functions
$f^j\in L^2(G)$ and the diffusion sources by
$\Psi^j\in L^2_{\mathbb F}(0,T;H_0^1(G))$, $j=1,2$, then equality of the local
boundary fluxes implies
$f^1=f^2$ in $L^2(G)$ and
$\Psi^1=\Psi^2$ in $L^2_{\mathbb F}(0,T;H_0^1(G))$. Thus the local boundary
flux alone determines both a deterministic drift source and an adapted
function-valued diffusion source, without any observation of the terminal
state. This observation regime differs from that of \cite{Yuan2021},
where simultaneous determination of two sources uses observations at the final
time together with lateral-boundary data. We state
\cref{thm:uniqueness} in the characteristic-source form because the inverse
problem considered in this paper is geometric and this notation keeps the
presentation focused.
\end{remark}

We now prepare the proof of \cref{thm:uniqueness}. The first ingredient is the
boundary unique continuation property used to propagate zero Cauchy data from
the observed boundary patch.

\begin{lemma}
\label{lem:boundary-ucp}
Let $I=(t_0,t_1)\subset\mathbb R$, and let $c\in L^\infty(I\times G)$. Suppose
that $v\in H^1(I;L^2(G))\cap L^2(I;H^2(G))$ solves
$v_t-\Delta v-cv=0$ in $I\times G$. If $v=0$ and
$\partial v/\partial \nu=0$ on $I\times\Gamma$,
then $v=0$ in $I\times G$. The same conclusion holds for the backward equation
$-v_t-\Delta v-cv=0$
after the time reversal $s=t_1+t_0-t$.
\end{lemma}

This is the homogeneous Cauchy-data consequence of
\cite[Theorem~5.1]{Yamamoto2009}; the backward assertion follows by time
reversal.

The drift part of the uniqueness proof is reduced to a deterministic separated
inverse source problem. For this purpose, define $A$ on $L^2(G)$ by
$D(A)=H^2(G)\cap H_0^1(G)$ and $Au=\Delta u+c(x)u$.

\vspace{-2mm}
\begin{lemma}
\label{lem:det-inverse-source}
Let $c\in L^\infty(G)$ be real-valued and $ K \subset \subset G $. 
Assume $f\in L^2(G)$,
$\operatorname{supp} f\subset K$, and
$r\in W^{1,\infty}(0,T)$  satisfying
$r(0)\ne0$. 
If the solution of\vspace{-2mm}
\begin{equation}
\label{eq:det-inverse-source}
\begin{cases}
y_t=\Delta y+c(x)y+r(t)f
&\text{in }(0,T)\times G,\\
y=0&\text{on }(0,T)\times\partial G,\\
y(0)=0&\text{in }G,
\end{cases}\vspace{-2mm}
\end{equation}
satisfies $\dfrac{\partial y}{\partial \nu}=0$ on $(0,T)\times\Gamma$, then
$f=0$.
\end{lemma}

\vspace{-2mm}

Lemma \ref{lem:det-inverse-source} appears to be known, though a precise reference is not readily found in the literature. For completeness, we provide its proof here.

\vspace{-2mm}

\begin{proof}
Let $S(t)$ be the Dirichlet semigroup generated by $A$. The solution admits the
mild representation 
$y(t)=\int_0^t S(t-s)r(s)f\,ds$.  
The support condition $\operatorname{supp}f\subset K\subset\subset G$ and standard
parabolic smoothing away from $K$ imply
$g_f(t):=\frac{\partial S(t)f}{\partial \nu}\big|_\Gamma\in L^2(0,T;L^2(\Gamma))$, and the boundary trace of the Duhamel formula is\vspace{-2mm}
$$
\frac{\partial y(t)}{\partial \nu}\Big|_\Gamma
=\int_0^t r(s)g_f(t-s)\,ds.\vspace{-2mm}
$$
Thus $\frac{\partial y(t)}{\partial \nu}=0$ on $(0,T)\times\Gamma$ gives, in
$L^2(0,T;L^2(\Gamma))$,
$\int_0^t r(t-s)g_f(s)\,ds=0$.
Since $r\in W^{1,\infty}(0,T)$, differentiating this gives
$r(0)g_f(t)+\int_0^t r'(t-s)g_f(s)\,ds=0$
for a.e. $t\in(0,T)$. Hence\vspace{-2mm}
$$
\|g_f(t)\|_{L^2(\Gamma)}
\leq
\frac{\|r'\|_{L^\infty(0,T)}}{|r(0)|}
\int_0^t\|g_f(s)\|_{L^2(\Gamma)}\,ds.\vspace{-2mm}
$$
Gronwall's inequality yields
$\frac{\partial S(t)f}{\partial \nu}\Big|_\Gamma=0$
for a.e. $t\in(0,T)$.

Let $v(t)=S(t)f$. By parabolic regularization for analytic semigroups
\cite[Chapter~4, Section~4.1, Corollary~1.5]{Pazy1983}, for every
$\varepsilon\in(0,T)$,
$v\in H^1(\varepsilon,T;L^2(G))
\cap L^2(\varepsilon,T;H^2(G)\cap H_0^1(G))$, solves
$\partial_t v-\Delta v-cv=0$ in $(\varepsilon,T)\times G$, and satisfies
$\partial v/\partial \nu=0$ on $(\varepsilon,T)\times\Gamma$.
By \cref{lem:boundary-ucp}, $v=0$ in $(\varepsilon,T)\times G$. Since
$\varepsilon>0$ is arbitrary and $S(t)f\to f$ in $L^2(G)$ as $t\downarrow0$,
we obtain $f=0$.
\end{proof}

Let $\tau \in (0,T)$  and $ c (x) \in L^\infty(G) $. Let $\mathcal X:=H^1(0,T;L^2(\Gamma))\cap L^2(0,T;H^{3/2}(\Gamma))$. 
Consider the backward  parabolic equation\vspace{-2mm}
\begin{equation}
\label{eq:fixed-time-boundary-reachability-q}
\begin{cases}
- q_{t}=\Delta q+c(x)q&\text{in }(\tau,T)\times G,\\
q=\psi&\text{on }(\tau,T)\times\Gamma,\\
q=0&\text{on }(\tau,T)\times(\partial G\setminus\Gamma),\\
q(T)=0&\text{in }G.
\end{cases}\vspace{-2mm}
\end{equation}
Denote by $q(\cdot;\psi)$ the solution to \cref{eq:fixed-time-boundary-reachability-q} corresponding to the control $\psi$.

\begin{lemma} 
\label{lem:fixed-time-boundary-reachability} 
(1)
$\mathcal R(\tau ):=\{q(\tau;\psi): \psi\in C_c^\infty(( \tau, T )\times\Gamma)\}$  is dense in $L^2(G)$. 

(2)	There exists a countable set
$\{\psi_j\}_{j\ge1}\subset C_c^\infty((0,T)\times\Gamma)$
such that \vspace{-2mm}
\begin{align}\label{7.30-eq1}
\operatorname{span}\{q(t;\psi_j):j\ge1\} \mbox{ is dense in }
L^2(G),\qquad \forall t\in (0,T).
\end{align}
\end{lemma}

\begin{remark} 
The approximate controllability of \cref{eq:fixed-time-boundary-reachability-q} with boundary controls in $L^\infty((0,T)\times\Gamma)$ was shown in \cite{Fabre1995}, yielding the density of the reachable set in $L^2(G)$. Although the proof of part (1) of \cref{lem:fixed-time-boundary-reachability} is standard, it is not a direct consequence of that result, since $C_c^\infty((0,T)\times\Gamma)$ is not dense in $L^\infty((0,T)\times\Gamma)$. Part (2), on the other hand, appears to be new and is of independent interest for control theory.
\end{remark}
\begin{proof}
\textbf{Step 1.} We prove that $\mathcal R(\tau )$ is dense in $L^2(G)$ by contradiction. Assume the contrary. Then there exists $\phi\in L^2(G)$ orthogonal to $\mathcal R(\tau)$. Let $p$ solve\vspace{-2mm}
\begin{equation}
\label{eq:fixed-time-boundary-reachability-p}
\begin{cases}
p_{t}=\Delta p+c(x)p&\text{in }(\tau,T)\times G,\\
p=  0 &\text{on }(\tau,T)\times \partial G,\\
p(\tau)= \phi &\text{in }G.
\end{cases}\vspace{-2mm}
\end{equation}
By parabolic regularization for analytic semigroups, for every $\delta\in(0,T-\tau)$,
$p\in H^1(\tau+\delta,T;L^2(G))
\cap L^2(\tau+\delta,T;H^2(G)\cap H_0^1(G))$.
For $\psi\in C_c^\infty((\tau+\delta,T)\times\Gamma)$, multiplying the equation for $q(\cdot;\psi)$ by $p$ and integrating by parts yields\vspace{-2mm}
\begin{equation}\label{7.30-eq2}
0=(\phi,q(\tau;\psi))_{L^2(G)}
=-\int_{\tau+\delta}^T\int_\Gamma
\psi\,\frac{\partial p}{\partial \nu}\,d\sigma dt.\vspace{-2mm}
\end{equation}
Since $\delta>0$ is arbitrary and $C_c^\infty((\tau+\delta,T)\times\Gamma)$ is dense in $L^2((\tau+\delta,T)\times\Gamma)$, we deduce from \cref{7.30-eq2} that $\frac{\partial p}{\partial \nu}=0$ on $(\tau,T)\times\Gamma$. Applying \cref{lem:boundary-ucp} on $(\tau+\delta,T)$ yields $p=0$ in $(\tau+\delta,T)\times G$. Letting $\delta\downarrow0$ and using $p\in C([\tau,T];L^2(G))$, we obtain $\phi=p(\tau)=0$. Hence $\mathcal R(\tau)$ is dense in $L^2(G)$.

\textbf{Step 2.} We construct a single countable family that works uniformly for all fixed times.

Define
$\mathcal X:=\{\psi\in H^1(0,T;L^2(\Gamma))\cap L^2(0,T;H^{3/2}(\Gamma)): \psi=0 \text{ on } \partial\Gamma,\; \psi(0)=\psi(T)=0\}$.
Let $\mathcal J:=\{(t_1,t_2): 0\leq t_1< t_2 \leq T,\; t_1, t_2\in\mathbb{Q}\}$.
Since $\mathcal X$ is separable, for each $J\in\mathcal J$ we may select a countable subset $\{\psi^J_j\}_{j\ge1}\subset C_c^\infty((0,T)\times\Gamma)$ that is dense in $C_c^\infty(J\times\Gamma)$ with respect to the $\mathcal X$-topology. As $\mathcal J$ is countable, the set
$\{\psi_j\}_{j\ge1}:=\bigcup_{J\in \mathcal J}\{\psi^J_j\}_{j\ge1}$
is countable.

Now fix $t\in(0,T)$. For any $\psi\in C_c^\infty((t,T)\times\Gamma)$, its support lies in $J\times\Gamma$ for some $J\in\mathcal J$. Consequently, there exists a subsequence $\{\psi_{j_k}\}_{k\ge1}\subset\{\psi_j\}_{j\ge1}$ such that\vspace{-2mm}
$$
\lim_{k\to\infty}\|\psi_{j_k}-\psi\|_{H^1(0,T;L^2(\Gamma))\cap L^2(0,T;H^{3/2}(\Gamma))}=0.\vspace{-2mm}
$$
By parabolic boundary regularity theory (\cite[Chapter~4, Theorems~2.1 and~4.2]{LionsMagenes1972Vol2}), obtained via a boundary lifting argument and the standard $L^2$-regularity theory for the homogeneous Dirichlet problem, the map $\psi\mapsto q(t;\psi)$ is continuous from $\mathcal X$ to $L^2(G)$ for each fixed $t\in(0,T)$. Hence,\vspace{-2mm}
$$
\lim_{k\to\infty}\|q(t;\psi_{j_k}) - q(t;\psi)\|_{L^2(G)}=0.\vspace{-2mm}
$$
Combined with assertion (1), this establishes \cref{7.30-eq1}.
\end{proof}

We now combine the drift inverse-source lemma and the boundary reachability
statement to prove the uniqueness theorem.
\begin{proof}[Proof of \cref{thm:uniqueness}]
Set $w=u^1-u^2$, $f_0=\chi_{D_0^1}-\chi_{D_0^2}$, and $f_1=\chi_{D_1^1}-\chi_{D_1^2}$. Then $f_0,f_1\in L^2(G)$ and $\operatorname{supp}f_i\subset K$. By \cref{eqStateGIP,lem:forward-local-reg}, $w$ satisfies\vspace{-2mm}
\begin{equation}
\label{eq:uniqueness-state-diff}
\begin{cases}
d w=(\Delta w+f_0+a w)\,dt + (b w+f_1)\,dW(t) & \text{in } (0,T)\times G,\\[2pt]
w=0 & \text{on } (0,T)\times \partial G,\\[2pt]
w(0,x)=0 & \text{in } G,
\end{cases}\vspace{-2mm}
\end{equation}
and $\partial w/\partial\nu=0$ on $(0,T)\times\Gamma$.

\textbf{Step 1. Identification of the drift source.}
Define $\mu(t)=\exp\bigl(-\int_0^t a_0(s)\,ds\bigr)$ and $z=\mu w$. By It\^o's formula,
$dz=\mu\,dw+w\,d\mu=(Az+\mu f_0)\,dt+(bz+\mu f_1)\,dW(t)$,
so that\vspace{-2mm}
\begin{equation}
\label{eq:uniqueness-state-z}
\begin{cases}
d z=(Az+\mu f_0)\,dt + (b z+\mu f_1)\,dW(t) & \text{in } (0,T)\times G,\\[2pt]
z=0 & \text{on } (0,T)\times \partial G,\\[2pt]
z(0,x)=0 & \text{in } G.
\end{cases}\vspace{-2mm}
\end{equation}
Since $\mu$ is independent of $x$, we have $\partial z/\partial\nu=\mu\,\partial w/\partial\nu=0$ on $(0,T)\times\Gamma$.

Let $\zeta(t)=\mathbb{E}[z(t)]$ and $r(t)=\mathbb{E}[\mu(t)]$. Taking expectations in the variational formulation of \cref{eq:uniqueness-state-z} yields\vspace{-2mm}
\begin{equation}
\label{eq:mean-equation}
\begin{cases}
\partial_t \zeta=A\zeta+r(t)f_0 & \text{in } (0,T)\times G,\\[2pt]
\zeta=0 & \text{on } (0,T)\times \partial G,\\[2pt]
\zeta(0,x)=0 & \text{in } G .
\end{cases}\vspace{-2mm}
\end{equation}
Moreover, $\partial \zeta/\partial\nu=\mathbb{E}[\partial z/\partial\nu]=0$ on $(0,T)\times\Gamma$. Since $\mu(0)=1$ and $\mu'(t)=-a_0(t)\mu(t)$, we have $\mu\in L^\infty(\Omega;W^{1,\infty}(0,T))$, hence $r\in W^{1,\infty}(0,T)$ with $r(0)=1$. Applying \cref{lem:det-inverse-source} gives $f_0=0$ in $L^2(G)$.

\textbf{Step 2. Identification of the diffusion source.}
With $f_0=0$, set $F=bz+\mu f_1$. For any $\psi\in C_c^\infty((0,T)\times\Gamma)$, let $q(\cdot;\psi)$ solve\vspace{-2mm}
\begin{equation*}
\begin{cases}
- q_{t}=\Delta q+c(x)q & \text{in }(0,T)\times G,\\
q=\psi & \text{on }(0,T)\times\Gamma,\\
q=0 & \text{on }(0,T)\times(\partial G\setminus\Gamma),\\
q(T)=0 & \text{in }G.
\end{cases}\vspace{-2mm}
\end{equation*}
Integration by parts gives\vspace{-2mm}
$$
\int_0^T\int_G F(t)\, q(t;\psi)\,dx\,dW(t)
=-\int_0^T\int_\Gamma
\frac{\partial z}{\partial \nu}\,\psi\,d\sigma dt=0.\vspace{-2mm}
$$
Then It\^o's isometry yields\vspace{-2mm}
\begin{equation}
\label{eqIntFq}
\mathbb{E}\int_0^T
\Bigl|\int_G F(t)\, q(t;\psi)\,dx\Bigr|^2 dt=0.\vspace{-2mm}
\end{equation}
Let $\{\psi_j\}_{j=1}^\infty$ be the countable family constructed in \cref{lem:fixed-time-boundary-reachability}. For each $j$, \cref{eqIntFq} implies
$\int_G F(t,\omega)\, q(t;\psi_j)\,dx=0$
for a.e.\ $(t,\omega)$. Removing the countable union of exceptional sets, we obtain a set $\mathcal N\subset\Omega\times(0,T)$ with $(\mathbb{P}\otimes dt)(\mathcal N)=0$ such that, for all $(\omega,t)\notin\mathcal N$,
$\int_G F(t,\omega)\, q(t;\psi_j)\,dx=0$ for every $j\geq1$.
For any such fixed $t$, \cref{lem:fixed-time-boundary-reachability} asserts that $\operatorname{span}\{q(t;\psi_j):j\ge1\}$ is dense in $L^2(G)$. Hence $F(t,\omega)=0$ in $L^2(G)$ for a.e.\ $(t,\omega)$.

The equation for $z$ now reduces to $\partial_t z=Az$, $z|_{(0,T)\times\partial G}=0$, $z(0)=0$, hence $z\equiv 0$. Since $\mu$ is bounded away from zero, $F=bz+\mu f_1=0$ yields $f_1=0$. With $f_0=0$, the proof is complete. 
\end{proof}

\section{Existence of the minimizer}
\label{sec:existence-minimizer}

This section returns to the general coefficient framework of \cref{eqStateGIP}; in particular, \cref{assumption:uniqueness} is not required here. We prove the existence of a measurable minimizer $(D_0,D_1)\in\mathcal A_K$ for the shape optimization problem \cref{eqOptGIP}. The proof is based on the decomposition 
$\mathcal{J}(D_0,D_1)=\mathcal{J}_1(D_0,D_1)+\mathcal{J}_{2}(D_0,D_1)$, where\vspace{-2mm}
$$
\mathcal{J}_1(D_0,D_1)
=\frac12 \mathbb{E}\int_0^T \int_{\Gamma}
\biggl|\frac{\partial u}{\partial \nu}-g\biggr|^2
\,d\sigma\,dt,
\qquad
\mathcal{J}_{2}(D_0,D_1)
=\beta_0\mathcal{P}_G(D_0)+\beta_1\mathcal{P}_G(D_1).\vspace{-2mm}
$$

Consider the set of characteristic functions (see \cite[p. 32]{Sokolowski1992})
$\operatorname{Char}(G) = \{ \chi \in L^{2}(G) \mid \chi (1-\chi) = 0,
\text{ a.e. in } G \}$.
We also use the closed subset
$\operatorname{Char}_K(G)
=\{\chi_D\in\operatorname{Char}(G):\chi_D=0
\text{ a.e. in }G\setminus K\}$.

The perimeter term fits naturally into this framework. We therefore begin with the standard lower semicontinuity property of the perimeter functional.

\begin{lemma}[{\cite[Lemma 2.6, p.~33]{Sokolowski1992}}]
\label{lemmaLowerSemiContinuityPerimeter}
The functional $ \chi_D \mapsto \mathcal{P}_{G}(D) $ is lower semicontinuous on $ \operatorname{Char}(G) $ with respect to the $ L^{2}(G) $-topology.
\end{lemma}

We prove the continuity of the misfit term next.

\begin{proposition}
\label{propJ1Continuity}
The following functional is continuous on
$\operatorname{Char}_K(G)\times\operatorname{Char}_K(G)$ with respect to the
product $L^2(G)\times L^2(G)$ topology:\vspace{-2mm}
$$
(\chi_{D_0},\chi_{D_1}) \mapsto
\mathcal{J}_1(D_0,D_1)
:=
\frac{1}{2} \mathbb{E} \int_0^T \int_{\Gamma}\bigg(\frac{\partial u}{\partial \nu}-g\bigg)^2 \,d\sigma\,dt,\vspace{-2mm}
$$
where $ u $ solves \cref{eqStateGIP} corresponding to
$(\chi_{D_0},\chi_{D_1})$.
\end{proposition}

\begin{proof}[Proof of \cref{propJ1Continuity}]

Fix $(\chi_{D_0^1},\chi_{D_1^1})$ and
$(\chi_{D_0^2},\chi_{D_1^2})$ in
$\operatorname{Char}_K(G)\times\operatorname{Char}_K(G)$, and let $u^1,u^2$ be the
corresponding solutions to \cref{eqStateGIP}. Set $w:=u^{1}-u^{2}$.
Then $w$ solves\vspace{-2mm}
\begin{equation*}
\left\{\begin{aligned}
&d w =
\big(\Delta w+a w+\chi_{D_0^1}-\chi_{D_0^2}\big)\,dt
+\big(b w+\chi_{D_1^1}-\chi_{D_1^2}\big)\,dW(t)
&& \text{in } (0, T) \times G, \\
&w =0 && \text{on } (0, T) \times \partial G, \\
&w(0, x) =0 && \text{in } G.
\end{aligned}\right.\vspace{-2mm}
\end{equation*}
By \cref{lem:forward-local-reg}, applied with an open set
$M$ satisfying $K\subset M \subset\subset G$, there exists a constant $C_K>0$ such that\vspace{-2mm}
\begin{equation}
\label{eqJ1ContinuityBoundaryEstimate}
\bigg\|\frac{\partial w}{\partial \nu}\bigg\|_{L^2_{\mathbb F}(0,T;L^2(\partial G))}
\leq
C_K\big(
\|\chi_{D_0^1}-\chi_{D_0^2}\|_{H^{-1}(G)}
+
\|\chi_{D_1^1}-\chi_{D_1^2}\|_{H^{-1}(G)}
\big).\vspace{-2mm}
\end{equation}
Since $L^2(G)\hookrightarrow H^{-1}(G)$ continuously, the right-hand side is
bounded by the corresponding $L^2(G)$ norms. With the same set $M$, applying
\cref{lem:forward-local-reg} directly to $u^j$, with initial datum $u_0$ and
sources $(\chi_{D_0^j},\chi_{D_1^j})$, and using
$\|\chi_{D_i^j}\|_{L^2(G)}\leq |G|^{1/2}$ (the Lebesgue measure), yields the uniform bound\vspace{-2mm}
\begin{equation}
\label{eqJ1UniformBoundaryEstimate}
\bigg\|\frac{\partial u^j}{\partial \nu}\bigg\|_{L^2_{\mathbb F}(0,T;L^2(\partial G))}
\leq C_K \big( |G|^{1/2} + \|u_0\|_{H^{1}_{0}(G)} \big),
\qquad j=1,2.\vspace{-2mm}
\end{equation}

By the Cauchy--Schwarz inequality, we obtain
\begin{align*}
|\mathcal{J}_1(D_0^1,D_1^1)-\mathcal{J}_1(D_0^2,D_1^2)|
& =
\frac{1}{2} \bigg|\mathbb{E} \int_{0}^{T} \int_{\Gamma}
\bigg[\bigg(\frac{\partial u^{1}}{\partial \nu}-g\bigg)^2-\bigg(\frac{\partial u^{2}}{\partial \nu}-g\bigg)^2\bigg] \,d\sigma\,dt\bigg|
\\
& \leq \frac{1}{2}
\bigg\|\frac{\partial u^{1}}{\partial \nu}+\frac{\partial u^{2}}{\partial \nu}-2 g\bigg\|_{L^{2}_{\mathbb{F}}(0, T; L^{2}(\Gamma))}
\bigg\|\frac{\partial w}{\partial \nu}\bigg\|_{L^{2}_{\mathbb{F}}(0, T; L^{2}(\Gamma))}.
\end{align*}
Combining this estimate with
\cref{eqJ1ContinuityBoundaryEstimate,eqJ1UniformBoundaryEstimate}, we obtain
\begin{align*}
&
|\mathcal{J}_1(D_0^1,D_1^1)-\mathcal{J}_1(D_0^2,D_1^2)|
\leq
C_K(u_{0}, g)
\left(
\|\chi_{D_0^1}-\chi_{D_0^2}\|_{L^{2}(G)}
+
\|\chi_{D_1^1}-\chi_{D_1^2}\|_{L^{2}(G)}
\right).
\end{align*}
Hence $\mathcal{J}_1(\cdot,\cdot)$ is continuous on
$\operatorname{Char}_K(G)\times\operatorname{Char}_K(G)$ with respect to the
product $L^{2}(G)\times L^2(G)$ topology. This completes the proof of
\cref{propJ1Continuity}.
\end{proof}

Combining \cref{lemmaLowerSemiContinuityPerimeter,propJ1Continuity,eqOptGIP}
gives the following lower semicontinuity result.
\begin{proposition}
\label{propLowerSemiContinuityJ}
The shape functional $\mathcal{J}(\cdot,\cdot)$ is lower semicontinuous on the set
$ \operatorname{Char}_K(G)\times\operatorname{Char}_K(G) \subset
L^{2}(G)\times L^2(G)$.
\end{proposition}
To pass from lower semicontinuity to existence, it remains to recover compactness for minimizing sequences. This is provided by the classical compactness theorem for characteristic functions with uniformly bounded perimeter.

\begin{lemma}[{\cite[Lemma 2.8, p.~34]{Sokolowski1992}}]
\label{lemmaCompactnessChar}
Let $ G $ be a bounded domain in $ \mathbb{R}^n $. For any $ R > 0 $, the set
$\operatorname{Char}(G, R) := \{ \chi_D \in \operatorname{Char}(G) \mid
\mathcal{P}_G(D) \leq R \}$
is compact in $ L^2(G) $.
\end{lemma}

The existence result is the following.

\begin{theorem}
\label{thmExistenceMinimizer}
There exists a measurable pair
$(\widehat D_0,\widehat D_1)\in\mathcal A_K$ such that
$\mathcal{J}(\widehat D_0,\widehat D_1)\leq\mathcal{J}(D_0,D_1)$
for all $(D_0,D_1)\in\mathcal A_K$.
\end{theorem}
\begin{proof}[Proof of \cref{thmExistenceMinimizer}]
Let $j_0=\mathcal J(\emptyset,\emptyset)$. By \cref{eqBoundaryObservationWellDefined}, $j_0<+\infty$. Since $\mathcal J\geq0$ on $\mathcal A_K$, the case $j_0=0$ trivially yields $(\emptyset,\emptyset)$ as a minimizer. We therefore assume $j_0>0$ in what follows.

Set $R_0:=j_0/\beta_0$ and $R_1:=j_0/\beta_1$, and introduce the truncated admissible class
\begin{align*}
\mathcal K_R
:=
\bigl(\operatorname{Char}_K(G)\cap \operatorname{Char}(G,R_0)\bigr)
\times
\bigl(\operatorname{Char}_K(G)\cap \operatorname{Char}(G,R_1)\bigr).
\end{align*}
As before, we identify each admissible set with its characteristic function. We claim that minimizing over $\mathcal A_K$ is equivalent to minimizing over $\mathcal K_R$, i.e.,
\begin{align}
\label{eq:inf-reduction-compact-class}
\inf_{(D_0,D_1)\in\mathcal A_K}\mathcal J(D_0,D_1)
=
\inf_{(\chi_{D_0},\chi_{D_1})\in\mathcal K_R}\mathcal J(D_0,D_1).
\end{align}
Since $\mathcal K_R\subset\operatorname{Char}_K(G)\times\operatorname{Char}_K(G)$, the inequality
\begin{align*}
\inf_{(D_0,D_1)\in\mathcal A_K}\mathcal J(D_0,D_1)
\leq
\inf_{(\chi_{D_0},\chi_{D_1})\in\mathcal K_R}\mathcal J(D_0,D_1)
\end{align*}
is immediate. Suppose $(D_0,D_1)\in\mathcal A_K$ does not belong to $\mathcal K_R$. Then either $\mathcal P_G(D_0)>R_0$ or $\mathcal P_G(D_1)>R_1$. Because $\mathcal J_1(\cdot,\cdot)\geq0$, the first case gives\vspace{-2mm}
$$
\mathcal J(D_0,D_1)
\geq
\beta_0\mathcal P_G(D_0)
>
j_0
=
\mathcal J(\emptyset,\emptyset),\vspace{-2mm}
$$
and the second case is analogous. Since $(\chi_\emptyset,\chi_\emptyset)\in\mathcal K_R$, no admissible pair outside $\mathcal K_R$ can lower the infimum below that over $\mathcal K_R$. This establishes \cref{eq:inf-reduction-compact-class}.

It remains to minimize over $\mathcal K_R$. By \cref{lemmaCompactnessChar}, $\operatorname{Char}(G,R_i)$ is compact in $L^2(G)$ for $i=0,1$. The subset $\operatorname{Char}_K(G)$ is closed in $L^2(G)$, because the condition $\chi=0$ a.e.\ on $G\setminus K$ is preserved under $L^2(G)$-convergence. Consequently, each intersection $\operatorname{Char}_K(G)\cap \operatorname{Char}(G,R_i)$ is compact in $L^2(G)$, and hence $\mathcal K_R$ is compact in $L^2(G)\times L^2(G)$.

Let $(D_0^k,D_1^k)$ be a minimizing sequence in $\mathcal K_R$. By compactness, passing to a subsequence if necessary, there exists $(\chi_{\widehat D_0},\chi_{\widehat D_1})\in\mathcal K_R$ such that $\chi_{D_0^k}\to \chi_{\widehat D_0}$ and $\chi_{D_1^k}\to \chi_{\widehat D_1}$ in $L^2(G)$. The lower semicontinuity established in \cref{propLowerSemiContinuityJ} then yields\vspace{-2mm}
$$
\mathcal J(\widehat D_0,\widehat D_1)
\leq
\liminf_{k\to\infty}\mathcal J(D_0^k,D_1^k)
=
\inf_{(\chi_{D_0},\chi_{D_1})\in\mathcal K_R}\mathcal J(D_0,D_1).\vspace{-2mm}
$$
Combined with \cref{eq:inf-reduction-compact-class}, this shows that $(\widehat D_0,\widehat D_1)\in\mathcal A_K$ is a minimizer over the original admissible class. This completes the proof of \cref{thmExistenceMinimizer}.
\end{proof}

\vspace{-4mm}
\section{Preliminaries on shape calculus}
\label{sec:shape-preliminaries}

This section recalls the notions from shape calculus used below. Details can be found in \cite{Sokolowski1992}. 

From this
section on, let $D \subset\subset G$ be an open set with Lipschitz boundary.
Let $\nu$ be the unit outward normal vector on $\partial D$. When the first
variation of the perimeter is used later, we further assume that $\partial D$ is
of class $C^2$.

Fix $V \in W_0^{1,\infty}(G; \mathbb{R}^n)$. For
$|\varepsilon|$ sufficiently small, define the perturbation of identity by
$F_{\varepsilon}(x)=x+\varepsilon V(x)$, $x\in G$.
For the shape $D$, define the perturbed domain\vspace{-2mm}
\begin{equation}
\label{eqPerturbedDomain}
D_{\varepsilon} = F_{\varepsilon}(D) \deq \{F_{\varepsilon}(x) \in \mathbb{R}^n \mid x \in D\}.\vspace{-2mm}
\end{equation}
Since $V$ is Lipschitz and $D \subset\subset G$, there exists $\varepsilon_0 > 0$ such that, for all $|\varepsilon| < \varepsilon_0$, the map $F_{\varepsilon}$ is bi-Lipschitz on $G$, the perturbed set $D_{\varepsilon}$ remains an open Lipschitz subdomain of $G$, and
$F_{\varepsilon}(\partial D)=\partial D_{\varepsilon}$.
This is the perturbation of identity method; see \cite{Haslinger2003,Delfour2011}.

\begin{definition}[Material derivative and shape derivative]
\label{defShapeDerivative}
Let $ X(D) $ be a Banach space of functions on $ D $, and let $ f_{\varepsilon} : D_{\varepsilon} \to \mathbb{R} $ be a family of functions such that $ f = f_0 $. The material derivative of $ f $ at $ D $ in the direction $ V $ is defined by
$\dot{f}=\lim\limits_{\varepsilon \rightarrow 0+}
(f_{\varepsilon} \circ F_{\varepsilon} - f)/\varepsilon$,
whenever the limit exists in $ X(D) $. If, in addition, $ f \in W^{1,1}_{\operatorname{loc}}(D) $ and $ \dot{f} $ exists, then the shape derivative of $ f $ at $ D $ in the direction $ V $ is defined by
$f'=\dot{f}-\nabla f\cdot V$,
whenever the right-hand side is well defined in the corresponding function space.
\end{definition}

In the present problem, when a quantity is already defined on the fixed domain $G$, its derivative is understood directly in the corresponding function space over $G$.

\begin{definition}
\label{defEulerianDerivative}
The Eulerian derivative of a shape functional $ J(D) $ at $ D $ in the direction $ V $ is defined as
$d J(D;V)=\lim\limits_{\varepsilon \rightarrow 0+}
(J(D_{\varepsilon}) - J(D))/\varepsilon$,
if the limit exists.
\end{definition}
\vspace{-2mm}
\begin{lemma}[{\cite[Theorem 4.1, p.~483]{Delfour2011}}]
\label{thmHadamard}
Let $ \tau > 0 $. Assume that
$\Phi  \in C\big([0,\tau); W^{1,1}_{\operatorname{loc}}(G)\big)
\cap C^1\big([0,\tau); L^1_{\operatorname{loc}}(G)\big)$. Let
$D \subset G$ be measurable, and define
$J(\varepsilon)=\int_{D_{\varepsilon}} \Phi(\varepsilon,x)\,dx$.
Then $ J $ admits a right derivative at $ \varepsilon = 0 $, and\vspace{-2mm}
$$
J'(0+) = \int_D \Big[ \partial_{\varepsilon} \Phi (0, x) + \operatorname{div}\big(\Phi (0, x) V(x)\big) \Big] \, dx.\vspace{-2mm}
$$
If, in addition, $ D $ is open with Lipschitz boundary, then\vspace{-2mm}
$$
J'(0+) = \int_D \partial_{\varepsilon} \Phi (0, x) \, dx + \int_{\partial D} \Phi (0, x) \, V \cdot \nu \, d \sigma.\vspace{-2mm}
$$
\end{lemma}

The following proposition is a direct consequence of \cref{thmHadamard}.
\begin{proposition}
\label{propChiDe}
Assume that $ D $ is open with Lipschitz boundary. Then the shape derivative $ \chi_{D}' $ of $ \chi_D $ satisfies\vspace{-2mm}
\begin{equation}
\label{eqChiDe}
\langle \chi_D', \varphi \rangle_{H^{-1}(G), H_{0}^{1}(G)}
=
\int_{\partial D} \varphi \, V \cdot \nu \, d \sigma,
\qquad
\forall \, \varphi \in H_{0}^{1}(G).\vspace{-2mm}
\end{equation}
\end{proposition}

\section{Shape sensitivity analysis}
\label{secShapeSensitivityAnalysis}

In this section, we derive a local Eulerian derivative formula for smooth
interior shapes. This local sensitivity analysis is carried out on a smooth
subclass of admissible shapes and is not claimed to be an Euler--Lagrange
formula for arbitrary finite-perimeter minimizers obtained in
\cref{thmExistenceMinimizer}. Throughout the section, unless otherwise stated, we assume
that $D_0,D_1\subset\subset G$ are open with $C^{2}$ boundaries and that
$V_0,V_1\in C_{c}^{1}(G;\mathbb{R}^n)$. Let $\nu_i$ be the unit
outward normal vector on $\partial D_i$, and let
$\rho_i=\operatorname{div}_{\partial D_i}\nu_i$ be the scalar mean curvature,
$i=0,1$. When this local formula is used for the constrained class
$\mathcal A_K$, the perturbations are restricted so that
$D_i^\varepsilon\subset K$ for all sufficiently small $\varepsilon>0$.

For $i=0,1$ and $|\varepsilon|$ small, define
$F_{\varepsilon}^{i}(x)=x+\varepsilon V_i(x)$ and
$D_i^\varepsilon=F_\varepsilon^i(D_i)$.
Since $D_i\subset\subset G$, we have $D_i^\varepsilon\subset\subset G$ for all
sufficiently small $\varepsilon>0$. 

Let $u=u^{D_0,D_1}$ be the solution to \cref{eqStateGIP} associated with $(D_0,D_1)$. From
\cref{eqBoundaryObservationWellDefined}, we set\vspace{-2mm}
\begin{equation}
\label{eq:boundary-residual}
r:=\frac{\partial u}{\partial \nu}\bigg|_\Gamma-g
\in L^2_{\mathbb F}(0,T;L^2(\Gamma)),\vspace{-2mm}
\end{equation}
and  $ \widetilde{r} =   r(t,x)  $ on $ \Gamma $ and $ \widetilde{r} = 0 $ on $ \partial G \setminus \Gamma $. 

Consider the following adjoint equation:\vspace{-2mm}
\begin{equation}
\label{eqAdjGIP}
\begin{cases}
d p=-(\Delta p+a p+b q)\,dt+q\,dW(t) & \text{in } (0,T)\times G,\\
p=\widetilde r & \text{on } (0,T)\times \partial G,\\
p(T)=0 & \text{in } G.
\end{cases}\vspace{-2mm}
\end{equation}

In the following definition, the forward test state is the strong solution
provided by \cref{thmWellPosedness}.

\begin{definition}
\label{defAdjoint}
A pair of processes
$(p,q)\in L^2_{\mathbb F}(0,T;L^2(G))
\times L^2_{\mathbb F}(0,T;H^{-1}(G))$
is called a transposition solution to \cref{eqAdjGIP} if, for every
$\Phi\in L^2_{\mathbb F}(0,T;L^2(G))$ and every
$\Psi\in L^2_{\mathbb F}(0,T;H_0^1(G))$, one has\vspace{-2mm}
\begin{equation}
\label{eqAdjointTransposition}
\mathbb E\int_0^T\int_G p \Phi \,dx\,dt
+
\mathbb E\int_0^T
\langle q,\Psi\rangle_{H^{-1}(G),H_0^1(G)}\,dt
=
-\mathbb E\int_0^T\int_{\Gamma}
r\, \frac{\partial y^{\Phi,\Psi}}{\partial \nu}\,d\sigma\,dt,\vspace{-2mm}
\end{equation}
where $y^{\Phi,\Psi}$ is the strong solution of \cref{eqStateGeneral} with
zero initial datum and sources $(\Phi,\Psi)$.
\end{definition}
The following well-posedness result is an immediate corollary of  the transposition solution theory for
backward stochastic parabolic equations with nonhomogeneous boundary data
\cite[Theorem 7.14]{Lue2021a} and the regularity theory for stochastic parabolic equations \cite[Theorem 4.11]{Lue2021a}.
\begin{proposition}
\label{propAdjointWellPosedness}
For every $r\in L^2_{\mathbb F}(0,T;L^2(\Gamma))$, there exists a
unique transposition solution $(p,q)$ to \cref{eqAdjGIP}. Moreover, there
exists a constant $C>0$ such that\vspace{-2mm}
\begin{equation}
\label{eqAdjointEstimate}
\|p\|_{L^2_{\mathbb F}(0,T;L^2(G))}
+
\|q\|_{L^2_{\mathbb F}(0,T;H^{-1}(G))}
\leq
C \|r\|_{L^2_{\mathbb F}(0,T;L^2(\Gamma))}.\vspace{-2mm}
\end{equation}
\end{proposition}
\begin{proof}
Set
$\mathfrak X
:=
L^2_{\mathbb F}(0,T;L^2(G))
\times
L^2_{\mathbb F}(0,T;H_0^1(G))$ 
equipped with the natural inner product\vspace{-2mm}
$$
((\Phi_1,\Psi_1),(\Phi_2,\Psi_2))_{\mathfrak X}
=
\mathbb E\int_0^T\int_G \Phi_1\Phi_2\,dx\,dt
+
\mathbb E\int_0^T
(\Psi_1,\Psi_2)_{H_0^1(G)}\,dt.\vspace{-2mm}
$$
For $(\Phi,\Psi)\in\mathfrak X$, let $y^{\Phi,\Psi}$ be the strong solution of
\cref{eqStateGeneral} with $y_0=0$. By \cref{thmWellPosedness} and the trace
theorem,\vspace{-2mm}
$$
\bigg\|\frac{\partial y^{\Phi,\Psi}}{\partial \nu}\bigg\|_
{L^2_{\mathbb F}(0,T;L^2(\partial G))}
\le
C\big(
\|\Phi\|_{L^2_{\mathbb F}(0,T;L^2(G))}
+
\|\Psi\|_{L^2_{\mathbb F}(0,T;H_0^1(G))}
\big).\vspace{-2mm}
$$
Define a linear functional $\Lambda_r$ on $\mathfrak X$ by\vspace{-2mm}
$$
\Lambda_r(\Phi,\Psi)
:=
-\mathbb E\int_0^T\int_{\Gamma}
r\,\frac{\partial y^{\Phi,\Psi}}{\partial \nu}\,d\sigma\,dt.\vspace{-2mm}
$$
The preceding estimate implies that $\Lambda_r$ is bounded on $\mathfrak X$, with\vspace{-2mm}
$$
\|\Lambda_r\|_{\mathfrak X'}
\le
C\|r\|_{L^2_{\mathbb F}(0,T;L^2(\Gamma))}.\vspace{-2mm}
$$
The Riesz type representation theorem (\cite[Theorem 2.55]{Lue2021a}) yields
unique elements
$p\in L^2_{\mathbb F}(0,T;L^2(G))$ and
$q\in L^2_{\mathbb F}(0,T;H^{-1}(G))$
such that\vspace{-2mm}
\begin{equation}
\label{eq:riesz-rep}
\Lambda_r(\Phi,\Psi)
=
\mathbb E\int_0^T\int_G p\,\Phi\,dx\,dt
+
\mathbb E\int_0^T
(q,\Psi)_{H^{-1}(G),H_0^1(G)}\,dt
\qquad\forall\,(\Phi,\Psi)\in\mathfrak X.\vspace{-2mm}
\end{equation}
The estimate \cref{eqAdjointEstimate} follows
directly from the operator norm bound of $\Lambda_r$ and the isometric
identification above. Uniqueness is immediate: if $(p,q)$ is another solution,
subtracting the corresponding transposition equations shows that $p=0$ and
$q=0$ by testing against arbitrary $(\Phi,\Psi)\in\mathfrak X$.
\end{proof}

Recall the definition of the Eulerian derivative of a shape functional in
\cref{defEulerianDerivative}. The main result of this section is the following
explicit formula for the shape derivative of the cost
functional $\mathcal{J}$, where
\begin{align*}
d\mathcal J(D_0,D_1;V_0,V_1)
=
\lim_{\varepsilon\to0+}
\frac{\mathcal J(D_0^\varepsilon,D_1^\varepsilon)
-\mathcal J(D_0,D_1)}{\varepsilon}.
\end{align*}
\vspace{-3mm}
\begin{theorem}[Shape derivative of the functional]
\label{thm:shape-derivative-main}
For every $V_0,V_1\in C_{c}^{1}(G;\mathbb{R}^n)$, the Eulerian derivative of
$\mathcal J$ exists and is given by\vspace{-2mm}
\begin{align}
\label{eq:shape-derivative-main}
\nonumber
d\mathcal J(D_0,D_1;V_0,V_1)
&=
-\mathbb E\int_0^T\int_{\partial D_0}
p(t)\,V_0\cdot \nu_0\,d\sigma\,dt
+
\beta_0\int_{\partial D_0}\rho_0\,V_0\cdot \nu_0\,d\sigma
\\
&\quad
-\mathbb E\int_0^T\int_{\partial D_1}
q(t)\,V_1\cdot \nu_1\,d\sigma\,dt
+
\beta_1\int_{\partial D_1}\rho_1\,V_1\cdot \nu_1\,d\sigma,
\end{align}
where $(p,q)$ is the transposition solution of the adjoint equation
\cref{eqAdjGIP}.
\end{theorem}

The proof is organized according to the following strategy.
\begin{enumerate}
\item We identify the geometric source derivatives in $H^{-1}(G)$ and use
\cref{lem:forward-local-reg} to control the corresponding boundary normal
traces.
\item We prove the local $H^1$ regularity needed to give meaning to the traces
of $p$ and $q$ on the moving interfaces.
\item We combine these two ingredients to prove
\cref{thm:shape-derivative-main}.
\end{enumerate}

\vspace{-2mm}

\subsection{Geometric source convergence}

The following lemma is a  consequence of
the weak $H^{-1}$-differentiability of transported $L^2$ functions; see
\cite[Chapter~2, Proposition~2.39, pp.~72--74]{Sokolowski1992}.  
We include the proof for completeness.

\begin{lemma}[Local convergence of a geometric source]
\label{lem:xi-eps-Hminus1}
Let $D\subset\subset G$ be open with Lipschitz boundary, let
$V\in C_c^1(G;\mathbb{R}^n)$, and let
$\overline D\subset M\subset\subset G$,
where $M$ is an open bounded Lipschitz set. If
$D_\varepsilon=(\operatorname{Id}+\varepsilon V)(D)$ and
$\xi_\varepsilon=(\chi_{D_\varepsilon}-\chi_D)/\varepsilon$, then there exists
$\varepsilon_0>0$ such that
$\operatorname{supp}\xi_\varepsilon\subset\subset M$ for every
$0<\varepsilon<\varepsilon_0$. Moreover, when $\xi_\varepsilon$ is regarded as
an element of $H^{-1}(G)$ by\vspace{-2mm}
$$
\langle \xi_\varepsilon,\varphi\rangle_{H^{-1}(G),H_0^1(G)}
:=
\int_G \xi_\varepsilon \varphi\,dx,
\qquad \varphi\in H_0^1(G),\vspace{-2mm}
$$
one has\vspace{-2mm}
\begin{equation}
\label{eq:xi-local-bound}
\|\xi_\varepsilon\|_{H^{-1}(G)}\le C,\vspace{-2mm}
\end{equation}
where $C$ is independent of $\varepsilon$, and\vspace{-2mm}
\begin{equation}
\label{eq:xi-local-weak}
\xi_\varepsilon\rightharpoonup \chi_D'
\quad\text{weakly in }H^{-1}(G)
\quad\text{as }\varepsilon\to0+.\vspace{-2mm}
\end{equation}
\end{lemma}

\begin{proof}[Proof of \cref{lem:xi-eps-Hminus1}] 
\textbf{Step 1. Support confinement.}
Let $\delta:=\operatorname{dist}(\overline D,\partial M)>0$. Since $F_\varepsilon\to \operatorname{Id}$ uniformly on $\overline D$, there exists $\varepsilon_0>0$ such that\vspace{-3mm}
$$
\sup_{x\in\overline D}|F_\varepsilon(x)-x|<\frac{\delta}{2},
\qquad 0<\varepsilon<\varepsilon_0.\vspace{-2mm}
$$
Hence $D_\varepsilon\subset\subset M$ for all $0<\varepsilon<\varepsilon_0$, and therefore\vspace{-2mm}
\begin{align*}
\operatorname{supp}\xi_\varepsilon
\subset \overline{D_\varepsilon}\cup \overline D\subset\subset M.
\end{align*}

\textbf{Step 2. Uniform $H^{-1}$ bound.}
For each fixed $\varepsilon\in(0,\varepsilon_0)$, $\xi_\varepsilon\in L^2(G)$ and consequently defines an element of $H^{-1}(G)$ via\vspace{-2mm}
$$
\langle \xi_\varepsilon,\varphi\rangle_{H^{-1}(G),H_0^1(G)}
:=
\int_G \xi_\varepsilon \varphi\,dx,
\qquad \varphi\in H_0^1(G).\vspace{-2mm}
$$
Fix an arbitrary $\varphi\in H_0^1(G)$ and set
$J_\varphi(\varepsilon):=\int_{D_\varepsilon}\varphi\,dx$.
Since $G$ is bounded, we have $H_0^1(G)\hookrightarrow W^{1,1}(G)$. Applying \cref{thmHadamard} with the integrand $\Phi(\varepsilon,x)=\varphi(x)$ (which is independent of $\varepsilon$) yields\vspace{-1mm}
\begin{align}
\label{eq:pointwise-limit}
\lim_{\varepsilon\to0+}
\langle \xi_\varepsilon,\varphi\rangle
=
J_\varphi'(0+)
=
\int_D \operatorname{div}(\varphi V)\,dx.
\end{align}
Thus, for each $\varphi\in H_0^1(G)$, the numerical sequence
$\{\langle \xi_\varepsilon,\varphi\rangle\}_{0<\varepsilon<\varepsilon_0}$ is convergent and hence bounded. This establishes pointwise boundedness of the family
$\{\xi_\varepsilon\}_{0<\varepsilon<\varepsilon_0}\subset H^{-1}(G)$
on the Banach space $H_0^1(G)$. By the uniform boundedness principle,\vspace{-2mm}
\begin{equation}
\label{eq:uniform-bound}
\sup_{0<\varepsilon<\varepsilon_0}\|\xi_\varepsilon\|_{H^{-1}(G)}<\infty,\vspace{-2mm}
\end{equation}
which is precisely \cref{eq:xi-local-bound}.

\textbf{Step 3. Identification of the weak limit.}
It remains to prove \cref{eq:xi-local-weak}. From \eqref{eq:pointwise-limit} and Green's formula on the Lipschitz domain $D$, we obtain, for every $\varphi\in H_0^1(G)$,\vspace{-2mm}
$$
\lim_{\varepsilon\to0+}
\langle \xi_\varepsilon,\varphi\rangle_{H^{-1}(G),H_0^1(G)}
=
\int_D \operatorname{div}(\varphi V)\,dx
=
\int_{\partial D}\varphi\,V\cdot \nu\,d\sigma
=
\langle \chi_D',\varphi\rangle_{H^{-1}(G),H_0^1(G)}.\vspace{-2mm}
$$
The uniform bound \eqref{eq:uniform-bound} guarantees that every subsequence of $\{\xi_\varepsilon\}_{0<\varepsilon<\varepsilon_0}$ contains a further subsequence converging weakly in $H^{-1}(G)$. By the uniqueness of the pointwise limit, every such weak limit must coincide with $\chi_D'$. Consequently, the entire family converges weakly:\vspace{-2mm}
$$
\xi_\varepsilon\rightharpoonup \chi_D'
\quad\text{weakly in }H^{-1}(G)
\quad\text{as }\varepsilon\to0+.\vspace{-2mm}
$$
This completes the proof of \cref{lem:xi-eps-Hminus1}.
\end{proof}

\subsection{Interior regularity of the adjoint}

\begin{proposition}
\label{thm:adjoint-interior-regularity}
Let $M  \subset\subset G$ be an open set with Lipschitz boundary, and let
$(p,q)$ be the transposition solution of \cref{eqAdjGIP}.
Then
$p|_M,q|_M\in L^2_{\mathbb F}(0,T;H^1(M))$
and\vspace{-2mm}
\begin{equation}
\label{eq:adjoint-interior-est}
\|p\|_{L^2_{\mathbb F}(0,T;H^1(M))}
+
\|q\|_{L^2_{\mathbb F}(0,T;H^1(M))}
\le C_M\|r\|_{L^2_{\mathbb F}(0,T;L^2(\Gamma))}.\vspace{-2mm}
\end{equation}
\end{proposition}

\begin{proof}[Proof of \cref{thm:adjoint-interior-regularity}]

The $L^2(G)$ bound for $p$ is already contained in \cref{eqAdjointEstimate}. We divide the remaining proof into four steps.

\smallskip
\noindent
\textbf{Step 1. $L^2(M)$-regularity of $q$.}
Let $\mathcal D_M$ denote the set of all finite linear combinations
$\varphi(t,\omega,x)=\sum\limits_{j=1}^N \alpha_j(t,\omega)\psi_j(x)$ ($N\in\mathbb N$),
where each $\alpha_j$ is a bounded progressively measurable simple process
and each $\psi_j\in C_c^\infty(M)$.
The set $\mathcal D_M$ is dense in $L^2_{\mathbb F}(0,T;L^2(M))$.
Taking $\Phi=0$ and $\Psi=\varphi\in\mathcal D_M$ in
\cref{eqAdjointTransposition} yields
\begin{align}
\label{eq:q-functional}
\mathbb E\int_0^T
\langle q,\varphi\rangle_{H^{-1}(G),H_0^1(G)}\,dt
=
-\mathbb E\int_0^T\int_{\Gamma}
r\,\frac{\partial y^{0,\varphi}}{\partial \nu}\,d\sigma\,dt .
\end{align}
The strong solution $y^{0,\varphi}$ is also a mild solution. Applying
\cref{lem:forward-local-reg} to this forward equation, we obtain
\begin{align}
\label{eq:q-estimate}
\bigg|
\mathbb E\int_0^T
\langle q,\varphi\rangle_{H^{-1}(G),H_0^1(G)}\,dt
\bigg|
&\leq
\|r\|_{L^2_{\mathbb F}(0,T;L^2(\Gamma))}
\bigg\|
\frac{\partial y^{0,\varphi}}{\partial \nu}
\bigg\|_{L^2_{\mathbb F}(0,T;L^2(\Gamma))}  \leq
C_M\|r\|_{L^2_{\mathbb F}(0,T;L^2(\Gamma))}
\|\varphi\|_{L^2_{\mathbb F}(0,T;L^2(M))}.
\end{align}

By density of $\mathcal D_M$ in $L^2_{\mathbb F}(0,T;L^2(M))$, the linear map
$\displaystyle
\varphi\mapsto
\mathbb E\int_0^T
\langle q,\varphi\rangle_{H^{-1}(G),H_0^1(G)}\,dt
$
extends uniquely to a bounded linear functional on
$L^2_{\mathbb F}(0,T;L^2(M))$.
The Riesz representation theorem furnishes an element
$\tilde q\in L^2_{\mathbb F}(0,T;L^2(M))$ such that\vspace{-2mm}
\begin{equation}
\label{eq:q-representation}
\mathbb E\int_0^T
\langle q,\varphi\rangle_{H^{-1}(G),H_0^1(G)}\,dt
=
\mathbb E\int_0^T\int_M
\tilde q\,\varphi\,dx\,dt
\qquad\forall\,\varphi\in\mathcal D_M.\vspace{-2mm}
\end{equation}
By density, \eqref{eq:q-representation} holds for all
$\varphi\in L^2_{\mathbb F}(0,T;L^2(M))$.
Since the embedding $L^2(M)\hookrightarrow H^{-1}(M)$ is continuous and
injective, the restriction $q|_M$ coincides with $\tilde q$ as an element of
$L^2_{\mathbb F}(0,T;L^2(M))$. In the sequel we denote this $L^2(M)$
representative again by $q$, and the estimate\vspace{-2mm}
\begin{equation}
\label{eq:q-L2-bound}
\|q\|_{L^2_{\mathbb F}(0,T;L^2(M))}
\leq
C_M\|r\|_{L^2_{\mathbb F}(0,T;L^2(\Gamma))}\vspace{-2mm}
\end{equation}
follows directly from \eqref{eq:q-estimate} and the Riesz representation.

\smallskip
\noindent
\textbf{Step 2. Weak derivatives of $p$.}
Let $k\in\{1,\ldots,n\}$ and let $\varphi\in\mathcal D_M$ be as above.
Taking $\Phi=\partial_{x_k}\varphi$ and $\Psi=0$ in
\cref{eqAdjointTransposition} gives\vspace{-2mm}
\begin{equation}
\label{eq:p-functional}
\mathbb E\int_0^T\int_M
p\,\partial_{x_k}\varphi\,dx\,dt
=
-\mathbb E\int_0^T\int_{\Gamma}
r\,\frac{\partial y^{\Phi,0}}{\partial \nu}\,d\sigma\,dt.\vspace{-1mm}
\end{equation}
Since $\operatorname{supp}\Phi\subset M$, applying
\cref{lem:forward-local-reg} yields\vspace{-2mm}
\begin{equation}
\label{eq:p-estimate}
\bigg|
\mathbb E\int_0^T\!\!\int_M
p \partial_{x_k}\varphi dx dt
\bigg|
\!\leq\!
\|r\|_{L^2_{\mathbb F}(0,T;L^2(\Gamma))} 
\bigg\|
\frac{\partial y^{\Phi,0}}{\partial \nu}
\bigg\|_{L^2_{\mathbb F}(0,T;L^2(\Gamma))} \!\leq\!
C_M\|r\|_{L^2_{\mathbb F}(0,T;L^2(\Gamma))}
\|\varphi\|_{L^2_{\mathbb F}(0,T;L^2(M))}.\vspace{-2mm}
\end{equation}
Consequently, the linear functional
$\displaystyle
\varphi\mapsto
-\mathbb E\int_0^T\int_M
p\,\partial_{x_k}\varphi\,dx\,dt
$
extends to a bounded linear functional on
$L^2_{\mathbb F}(0,T;L^2(M))$.
By the Riesz representation theorem, there exists
$p_k\in L^2_{\mathbb F}(0,T;L^2(M))$ such that\vspace{-2mm}
\begin{equation}
\label{eq:p-riesz}
-\mathbb E\int_0^T\int_M
p\,\partial_{x_k}\varphi\,dx\,dt
=
\mathbb E\int_0^T\int_M
p_k\,\varphi\,dx\,dt
\qquad\forall\,\varphi\in L^2_{\mathbb F}(0,T;L^2(M)).\vspace{-2mm}
\end{equation}

We now verify that $p_k$ is indeed the distributional derivative
$\partial_{x_k}p$ on $M$.
Fix $\psi\in C_c^\infty(M)$ and define the progressively measurable process\vspace{-2mm}
$$
A_\psi(t,\omega):=
\int_M p(t,\omega,x)\,\partial_{x_k}\psi(x)\,dx
+
\int_M p_k(t,\omega,x)\,\psi(x)\,dx,
\qquad (t,\omega)\in(0,T)\times\Omega.\vspace{-2mm}
$$
Since $p,p_k\in L^2_{\mathbb F}(0,T;L^2(M))$, we have
$A_\psi\in L^2(\Omega\times(0,T))$.
Taking $\varphi(t,\omega,x)=\alpha(t,\omega)\psi(x)$ in
\eqref{eq:p-riesz}, where $\alpha$ is an arbitrary bounded progressively
measurable simple process, yields\vspace{-2mm}
$$
\mathbb E\int_0^T \alpha(t,\omega)\,A_\psi(t,\omega)\,dt = 0.\vspace{-2mm}
$$
Bounded progressively measurable simple processes are dense in
$L^2(\Omega\times(0,T))$.
Hence $A_\psi=0$ for $(\mathbb P\otimes dt)$-a.e.\ $(\omega,t)$.

Now choose a countable subset
$\{\psi_m\}_{m=1}^\infty\subset C_c^\infty(M)$ that is dense in
$H_0^1(M)$. For each $m$, there exists a null set
$\mathcal N_m\subset\Omega\times(0,T)$ such that
$A_{\psi_m}(\omega,t)=0$ for all $(\omega,t)\notin\mathcal N_m$.
Set $\mathcal N:=\bigcup_{m=1}^\infty\mathcal N_m$. Then
$(\mathbb P\otimes dt)(\mathcal N)=0$, and for every
$(\omega,t)\notin\mathcal N$ we have\vspace{-2mm}
$$
\int_M p(t,\omega,x)\,\partial_{x_k}\psi_m(x)\,dx
=
-\int_M p_k(t,\omega,x)\,\psi_m(x)\,dx
\qquad\forall\,m\geq 1.\vspace{-2mm}
$$
For each such $(\omega,t)$, both sides are continuous linear functionals
of $\psi$ with respect to the $H_0^1(M)$-norm. By density of
$\{\psi_m\}_{m=1}^\infty$ in $H_0^1(M)$, the identity extends to all
$\psi\in H_0^1(M)$. In particular,
$p_k=\partial_{x_k}p$ in the distributional sense on $M$, and\vspace{-2mm}
\begin{align}
\label{eq:p-H1-bound}
\|p_k\|_{L^2_{\mathbb F}(0,T;L^2(M))}
\leq
C_M\|r\|_{L^2_{\mathbb F}(0,T;L^2(\Gamma))}.
\end{align}

\smallskip
\noindent
\textbf{Step 3. Weak derivatives of $q$.}
We repeat the argument of Step~2 for $q$.
Take $\Phi=0$ and $\Psi=\partial_{x_k}\varphi$ with
$\varphi\in\mathcal D_M$ in \cref{eqAdjointTransposition}.
Using the $L^2(M)$ representative of $q$ obtained in Step~1, we have\vspace{-2mm}
\begin{equation}
\label{eq:qk-functional}
\mathbb E\int_0^T\! \int_M
q\,\partial_{x_k}\varphi dx dt
=
-\mathbb E\int_0^T\! \int_{\Gamma}
r\,\frac{\partial y^{0,\Psi}}{\partial \nu} d\sigma dt.\vspace{-2mm}
\end{equation}
Applying \cref{lem:forward-local-reg} once more yields\vspace{-2mm}
\begin{equation}
\label{eq:qk-estimate}
\bigg|
\mathbb E\int_0^T\!\int_M
q\,\partial_{x_k}\varphi dx dt
\bigg|\!
\leq
\|r\|_{L^2_{\mathbb F}(0,T;L^2(\Gamma))}
\bigg\|
\frac{\partial y^{0,\Psi}}{\partial \nu}
\bigg\|_{L^2_{\mathbb F}(0,T;L^2(\Gamma))}\! \leq\!
C_M\|r\|_{L^2_{\mathbb F}(0,T;L^2(\Gamma))}
\|\varphi\|_{L^2_{\mathbb F}(0,T;L^2(M))}.\vspace{-2mm}
\end{equation}
By the same Riesz representation and density argument as in Step~2, there
exists $q_k\in L^2_{\mathbb F}(0,T;L^2(M))$ such that
$q_k=\partial_{x_k}q$ in the distributional sense on $M$, with\vspace{-2mm}
$$
\|q_k\|_{L^2_{\mathbb F}(0,T;L^2(M))}
\leq
C_M\|r\|_{L^2_{\mathbb F}(0,T;L^2(\Gamma))}.\vspace{-2mm}
$$

\smallskip
\noindent
\textbf{Step 4. Conclusion.}
Summing the estimates from Steps~1--3 over $k=1,\ldots,n$ and combining with
the $L^2(G)$ bound for $p$ from \cref{eqAdjointEstimate}, we obtain\vspace{-2mm}
$$
\|p\|_{L^2_{\mathbb F}(0,T;H^1(M))}
+
\|q\|_{L^2_{\mathbb F}(0,T;H^1(M))}
\leq
C_M\|r\|_{L^2_{\mathbb F}(0,T;L^2(\Gamma))},\vspace{-2mm}
$$
which completes the proof of \cref{thm:adjoint-interior-regularity}.
\end{proof}

\vspace{-3mm}

\subsection{Extension of the transposition identity to interior sources}

\begin{lemma}[Extension of the transposition identity to interior sources]
\label{lem:adjoint-interior-Hminus1-extension}
Let $M \subset\subset G$ be open with Lipschitz boundary, and let
$f_0,f_1\in L^2_{\mathbb F}(0,T;H^{-1}(G))$ satisfy
$\operatorname{supp}f_i(t,\omega)\subset \overline M$, $i=0,1$,
for a.e. $(t,\omega)$. Let $z$ be the mild solution  of \cref{eq:forward-hminus1} with $z_0=0$. If
$\eta\in C_c^\infty(G)$ satisfies $\eta=1$ in a neighborhood of
$\overline M$, then\vspace{-2mm}
\begin{equation}
\label{eq:adjoint-interior-Hminus1-extension}
\mathbb E\int_0^T\int_{\Gamma}
r\,\frac{\partial z}{\partial \nu}\,d\sigma\,dt
=
-\mathbb E\int_0^T
\langle f_0,\eta p\rangle_{H^{-1}(G),H_0^1(G)}\,dt 
-\mathbb E\int_0^T
\langle f_1,\eta q\rangle_{H^{-1}(G),H_0^1(G)}\,dt.\vspace{-2mm}
\end{equation}
\end{lemma}

\begin{proof}
Choose an open set $M_1$ such that
$\overline M \subset\subset M_1 \subset\subset G$ and
$\eta=1$ on $M_1$.
Choose adapted smooth approximations
$f_0^m,f_1^m\in L^2_{\mathbb F}(0,T;C_c^\infty(M_1))$ such that
$f_i^m\to f_i$ in $L^2_{\mathbb F}(0,T;H^{-1}(G))$, $i=0,1$.
This can be obtained by predictable simple-process approximation in time and
spatial mollification inside $M_1$. Let $z^m$ be the solution of
\cref{eq:forward-hminus1} with sources $(f_0^m,f_1^m)$ and zero initial datum.
For each $m$, the sources are admissible in \cref{defAdjoint}.   Hence
\begin{align*}
\mathbb E\int_0^T\int_{\Gamma}
r\,\frac{\partial z^m}{\partial \nu}\,d\sigma\,dt
&=
-\mathbb E\int_0^T\int_G p f_0^m\,dx\,dt
-
\mathbb E\int_0^T
\langle q,f_1^m\rangle_{H^{-1}(G),H_0^1(G)}\,dt
\\
&=
-\mathbb E\int_0^T
\langle f_0^m,\eta p\rangle_{H^{-1}(G),H_0^1(G)}\,dt
-\mathbb E\int_0^T
\langle f_1^m,\eta q\rangle_{H^{-1}(G),H_0^1(G)}\,dt.
\end{align*}
Here the last equality uses $\operatorname{supp}f_i^m\subset M_1$ and
$\eta=1$ on $M_1$, together with
\cref{thm:adjoint-interior-regularity}.

By \cref{eq:var-solution-bound}, $z^m\to z$ in
$L^2_{\mathbb F}(0,T;L^2(G))$ and in
$L^2_{\mathbb F}(\Omega;C([0,T];H^{-1}(G)))$. Applying
\cref{lem:forward-local-reg} to $z^m-z$ gives
$\partial z^m/\partial \nu\to\partial z/\partial \nu$ in
$L^2_{\mathbb F}(0,T;L^2(\partial G))$.
Since $\eta p,\eta q\in L^2_{\mathbb F}(0,T;H_0^1(G))$ by
\cref{thm:adjoint-interior-regularity}, the two duality pairings on the right
also converge. Passing to the limit proves
\cref{eq:adjoint-interior-Hminus1-extension}.
\end{proof}

\vspace{-2mm}
 
\subsection{Proof of the shape derivative formula}

\begin{theorem}[Shape derivative of the misfit]
\label{thm:shape-derivative-J1}
For every $V_0,V_1\in C_c^1(G;\mathbb{R}^n)$,\vspace{-2mm}
\begin{equation}
\label{eq:dJ1}
d\mathcal J_1(D_0,D_1;V_0,V_1)=
-\mathbb E\int_0^T\int_{\partial D_0}
p(t)\,V_0\cdot \nu_0\,d\sigma\,dt 
-\mathbb E\int_0^T\int_{\partial D_1}
q(t)\,V_1\cdot \nu_1\,d\sigma\,dt.\vspace{-2mm}
\end{equation}
Here $(p,q)$ denotes the transposition solution of
\cref{eqAdjGIP}.
\end{theorem}

\begin{proof}[Proof of \cref{thm:shape-derivative-J1}]

Let $0<\varepsilon\ll1$ and define\vspace{-2mm}
$$
\xi_\varepsilon^0:=\frac{\chi_{D_0^\varepsilon}-\chi_{D_0}}{\varepsilon},
\qquad
\xi_\varepsilon^1:=\frac{\chi_{D_1^\varepsilon}-\chi_{D_1}}{\varepsilon}.\vspace{-2mm}
$$
Let $u_\varepsilon$ be the solution to \cref{eqStateGIP} with
$(D_0,D_1)$ replaced by $(D_0^\varepsilon,D_1^\varepsilon)$, and set
$w_\varepsilon := (u_\varepsilon-u)/\varepsilon$.
Then $w_\varepsilon$ solves\vspace{-2mm}
\begin{equation}
\label{eq:w-eps-equation}
\begin{cases}
d w_\varepsilon
=
\bigl(\Delta w_\varepsilon+a w_\varepsilon+\xi_\varepsilon^0\bigr)\,dt
+\bigl(b w_\varepsilon+\xi_\varepsilon^1\bigr)\,dW(t)
& \text{in }(0,T)\times G,\\[4pt]
w_\varepsilon=0
& \text{on }(0,T)\times\partial G,\\[4pt]
w_\varepsilon(0)=0
& \text{in }G.
\end{cases}\vspace{-2mm}
\end{equation}

\smallskip
\noindent
\textbf{Step 1. Uniform estimate for the linearized state.}
Choose open bounded Lipschitz sets $M_0$ and $M$ such that 
$
\overline D_0\cup\overline D_1\subset M_0 \subset\subset M \subset\subset G$. 
By \cref{lem:xi-eps-Hminus1} applied to $D_0$ and $D_1$ with the enclosing set
$M_0$, there exists $\varepsilon_1>0$ such that for all
$0<\varepsilon<\varepsilon_1$,\vspace{-2mm}
$$
\operatorname{supp}\xi_\varepsilon^i\subset\subset M_0,\qquad
\|\xi_\varepsilon^i\|_{H^{-1}(G)}\le C_{M_0},\qquad i=0,1.\vspace{-2mm}
$$
In particular,\vspace{-2mm}
\begin{equation}
\label{eq:xi-Hminus1-bound}
\|\xi_\varepsilon^i\|_{L^2_{\mathbb F}(0,T;H^{-1}(G))}
\le
T^{1/2} C_{M_0},\qquad i=0,1.\vspace{-1mm}
\end{equation}
Applying \cref{lem:forward-local-reg} to the equation for $w_\varepsilon$ yields\vspace{-2mm}
\begin{equation}
\label{eq:w-eps-trace-bound}
\bigg\|\frac{\partial w_\varepsilon}{\partial \nu}\bigg\|_
{L^2_{\mathbb F}(0,T;L^2(\partial G))}
\le
C_{M_0}\Bigl(
\|\xi_\varepsilon^0\|_{L^2_{\mathbb F}(0,T;H^{-1}(G))}
+
\|\xi_\varepsilon^1\|_{L^2_{\mathbb F}(0,T;H^{-1}(G))}
\Bigr)
\le C_{M_0}'.\vspace{-2mm}
\end{equation}

\smallskip
\noindent
\textbf{Step 2. Expansion of the cost functional.}
Using the definition of $\mathcal J_1$ and the fact that
$u_\varepsilon = u + \varepsilon w_\varepsilon$, we expand\vspace{-3mm}
\begin{align*}
\frac{\mathcal J_1(D_0^\varepsilon,D_1^\varepsilon)-\mathcal J_1(D_0,D_1)}{\varepsilon}
& =
\frac{1}{2\varepsilon}\,
\mathbb E\int_0^T\int_{\Gamma}
\Bigl(
\Bigl|\frac{\partial u}{\partial \nu}+\varepsilon\frac{\partial w_\varepsilon}{\partial \nu}-g\Bigr|^2
-
\Bigl|\frac{\partial u}{\partial \nu}-g\Bigr|^2
\Bigr)\,d\sigma\,dt
\\
& =
\mathbb E\int_0^T\int_{\Gamma}
\Bigl(\frac{\partial u}{\partial \nu}-g\Bigr)\,
\frac{\partial w_\varepsilon}{\partial \nu}\,d\sigma\,dt
+
\frac{\varepsilon}{2}\,
\mathbb E\int_0^T\int_{\Gamma}
\Bigl|\frac{\partial w_\varepsilon}{\partial \nu}\Bigr|^2\,d\sigma\,dt.
\end{align*}
By \cref{eq:w-eps-trace-bound}, the second term is $O(\varepsilon)$. Denoting
$
r:=\frac{\partial u}{\partial \nu}-g\in L^2_{\mathbb F}(0,T;L^2(\Gamma))$, 
we obtain\vspace{-2mm}
\begin{equation}
\label{eq:J1-expansion}
\frac{\mathcal J_1(D_0^\varepsilon,D_1^\varepsilon)-\mathcal J_1(D_0,D_1)}{\varepsilon}
=
\mathbb E\int_0^T\int_{\Gamma}
r\,\frac{\partial w_\varepsilon}{\partial \nu}\,d\sigma\,dt
+ o(1)
\qquad(\varepsilon\to0+).\vspace{-2mm}
\end{equation}

\medskip
\noindent
\textbf{Step 3. Duality with the adjoint state.}
Choose $\eta\in C_c^\infty(M)$ with $\eta\equiv 1$ on a neighborhood of $\overline M_0$. By \cref{thm:adjoint-interior-regularity}, the transposition solution $(p,q)$ of \cref{eqAdjGIP} satisfies $p,q\in L^2_{\mathbb F}(0,T;H^1(M))$. Since $\operatorname{supp}\eta\subset M\subset\subset G$, extending $\eta p$ and $\eta q$ by zero outside $M$ yields elements of $L^2_{\mathbb F}(0,T;H_0^1(G))$.

We now invoke \cref{lem:adjoint-interior-Hminus1-extension}.
This lemma states that for any $f_0,f_1\in L^2(G)$ supported in $M_0$,
the following identity holds with the truncated adjoint
states:\vspace{-2mm}
\begin{equation}
\label{eq:duality-general}
\mathbb E\int_0^T\int_{\Gamma}
r\,\frac{\partial y^{f_0,f_1}}{\partial \nu}\,d\sigma\,dt
=
-\mathbb E\int_0^T
\langle f_0,\eta p(t)\rangle_{H^{-1}(G),H_0^1(G)}\,dt
-
\mathbb E\int_0^T
\langle f_1,\eta q(t)\rangle_{H^{-1}(G),H_0^1(G)}\,dt .\vspace{-2mm}
\end{equation}
Applying \eqref{eq:duality-general} with $f_0=\xi_\varepsilon^0$,
$f_1=\xi_\varepsilon^1$ and noting that $y^{\xi_\varepsilon^0,\xi_\varepsilon^1}=w_\varepsilon$,
we obtain\vspace{-2mm}
\begin{align}
\label{eq:shape-proof-xi-duality}
\mathbb E\int_0^T\int_{\Gamma}
r\,\frac{\partial w_\varepsilon}{\partial \nu}\,d\sigma\,dt
=
-\mathbb E\int_0^T
\langle \xi_\varepsilon^0,\eta p(t)\rangle_{H^{-1}(G),H_0^1(G)}\,dt
-
\mathbb E\int_0^T
\langle \xi_\varepsilon^1,\eta q(t)\rangle_{H^{-1}(G),H_0^1(G)}\,dt .
\end{align}

\smallskip
\noindent
\textbf{Step 4. Passage to the limit.}
We pass to the limit $\varepsilon\to0+$ in the two duality terms.
The arguments for the two terms are completely symmetric; we present the
details for the term involving $p$, the term involving $q$ being obtained
by replacing $(D_0,\xi_\varepsilon^0,p,V_0,\nu_0)$ with
$(D_1,\xi_\varepsilon^1,q,V_1,\nu_1)$.

\noindent
\textit{Pointwise convergence.}
By \cref{eq:xi-local-weak},
$\xi_\varepsilon^0\rightharpoonup \chi_{D_0}'$ weakly in $H^{-1}(G)$ as
$\varepsilon\to0+$.
For each fixed $(t,\omega)$,
$\eta p(t,\omega)\in H_0^1(G)$, and therefore\vspace{-2mm}
\begin{align}
\label{eq:pointwise-conv}
\lim_{\varepsilon\to0+}
\langle \xi_\varepsilon^0,\eta p(t,\omega)\rangle_{H^{-1}(G),H_0^1(G)}
=
\langle \chi_{D_0}',\eta p(t,\omega)\rangle_{H^{-1}(G),H_0^1(G)}.
\end{align}

\noindent
\textit{Domination.}
By \cref{eq:xi-local-bound} and the properties of $\eta$, we have the
pointwise estimate
\begin{equation}
\label{eq:domination}
\bigl|
\langle \xi_\varepsilon^0,\eta p(t,\omega)\rangle_{H^{-1}(G),H_0^1(G)}
\bigr|
\!\le\!
\|\xi_\varepsilon^0\|_{H^{-1}(G)}
\|\eta p(t,\omega)\|_{H_0^1(G)}\! \le\!
C_{M_0}\,
\|\eta p(t,\omega)\|_{H_0^1(G)} \!\le\!
C_M 
\|p(t,\omega)\|_{H^1(M)}.
\end{equation}
To verify the integrability of the dominating function, we compute\vspace{-2mm}
\begin{equation}
\label{eq:domination-integrable}
\mathbb E\int_0^T
\|p(t,\omega)\|_{H^1(M)}\,dt
\le
T^{1/2}\,
\bigl(\mathbb E\int_0^T
\|p(t,\omega)\|_{H^1(M)}^2\,dt\bigr)^{1/2} =
T^{1/2}\,
\|p\|_{L^2_{\mathbb F}(0,T;H^1(M))}
<\infty,\vspace{-2mm}
\end{equation}
where the finiteness follows from \cref{thm:adjoint-interior-regularity}.

\smallskip
\noindent
\textit{Identification of the limit.}
By \eqref{eq:pointwise-conv}, \eqref{eq:domination},
\eqref{eq:domination-integrable} and Lebesgue's dominated convergence theorem,\vspace{-2mm}
\begin{equation}
\label{eq:limit-p}
\lim_{\varepsilon\to0+}
\mathbb E\int_0^T
\langle \xi_\varepsilon^0,\eta p(t)\rangle_{H^{-1}(G),H_0^1(G)}\,dt
=
\mathbb E\int_0^T
\langle \chi_{D_0}',\eta p(t)\rangle_{H^{-1}(G),H_0^1(G)}\,dt.\vspace{-2mm}
\end{equation}
Since $\eta\equiv 1$ on a neighborhood of $\overline M_0\supset\partial D_0$,
\cref{eqChiDe} and the trace theorem on the Lipschitz boundary $\partial D_0$
give\vspace{-2mm}
\begin{equation}
\label{eq:limit-boundary-integral}
\langle \chi_{D_0}',\eta p(t,\omega)\rangle_{H^{-1}(G),H_0^1(G)}
=
\int_{\partial D_0}
p(t,\omega)\,V_0\cdot \nu_0\,d\sigma.\vspace{-2mm}
\end{equation}
Applying exactly the same reasoning to the term involving $q$, we obtain\vspace{-2mm}
\begin{equation}
\label{eq:limit-q}
\lim_{\varepsilon\to0+}
\mathbb E\int_0^T
\langle \xi_\varepsilon^1,\eta q(t)\rangle_{H^{-1}(G),H_0^1(G)}\,dt
=
\mathbb E\int_0^T\int_{\partial D_1}
q(t)\,V_1\cdot \nu_1\,d\sigma\,dt .
\end{equation}

\smallskip
\noindent
\textbf{Step 5. Conclusion.}
Combining \eqref{eq:J1-expansion}, \eqref{eq:shape-proof-xi-duality},
\eqref{eq:limit-p}--\eqref{eq:limit-q}, we obtain\vspace{-2mm}
\begin{align*}
d\mathcal J_1(D_0,D_1;V_0,V_1)
&=
\lim_{\varepsilon\to0+}
\frac{\mathcal J_1(D_0^\varepsilon,D_1^\varepsilon)-\mathcal J_1(D_0,D_1)}{\varepsilon}
\\
&=
-\mathbb E\int_0^T\int_{\partial D_0}
p(t)\,V_0\cdot \nu_0\,d\sigma\,dt
-
\mathbb E\int_0^T\int_{\partial D_1}
q(t)\,V_1\cdot \nu_1\,d\sigma\,dt.
\end{align*}
This completes the proof of \cref{thm:shape-derivative-J1}.
\end{proof}

\vspace{-2mm}

\begin{lemma}[{\cite[Theorem 4.3, p.~486]{Delfour2011}}]
\label{prop:perimeter-first-variation}
Let $D\subset\subset G$ be open with $C^{2}$ boundary, and let
$V\in C_{c}^{1}(G;\mathbb{R}^n)$.
Then the Eulerian derivative of $\mathcal P_G$ at $D$ in the direction $V$ is given by\vspace{-2mm}
\begin{equation*}
d\mathcal P_G(D;V)
=
\int_{\partial D}\rho\,V\cdot \nu\,d\sigma,\vspace{-2mm}
\end{equation*}
where $\nu$ is the outward unit normal to $D$ and
$\rho=\operatorname{div}_{\partial D}\nu$.
\end{lemma}
\vspace{-2mm}
\begin{proof}[Proof of \cref{thm:shape-derivative-main}]
Combining the decomposition
$\mathcal J=\mathcal J_1+\beta_0\mathcal P_G(D_0)+\beta_1\mathcal P_G(D_1)$
with \cref{thm:shape-derivative-J1} and applying
\cref{prop:perimeter-first-variation} to $(D_i,V_i)$, $i=0,1$, gives
\cref{eq:shape-derivative-main}.
\end{proof}

\vspace{-4mm}

\section{Numerical method}
\label{secNumericalAlgorithm}

This section constructs a finite-dimensional approximation of the
shape-gradient flow induced by the Eulerian derivative in
\cref{eq:shape-derivative-main}.  Let $h>0$, let
$\Delta t=T/N_t$, and let $M=N_{\rm MC}$ be the number of Monte Carlo
samples.  For polygonal shape pairs $(D_0^h,D_1^h)$, define the
sample-average functional\vspace{-2mm}
\begin{equation}
\label{eqDiscreteFunctional}
\mathcal J_{h,M}(D_0^h,D_1^h)
:=
\frac{1}{2M}
\sum_{m=1}^{M}
\sum_{n=0}^{N_t-1}
\Delta t
\sum_{x_j\in\Gamma_h}
w_j
\left|
\partial_n^h u_{m}^{n,h}(x_j)
-g_{m,j}^{\delta,n}
\right|^2
+\beta_0^h \mathcal P_h(D_0^h)
+\beta_1^h \mathcal P_h(D_1^h),\vspace{-2mm}
\end{equation}
where $w_j$ are the boundary quadrature weights, $\partial_n^h u$ is
the recovered discrete normal flux, and
$\mathcal P_h(D_i^h)$ is the polygonal perimeter.  The parameters
$\beta_i^h$ are the effective discrete perimeter weights.

The discretization below specifies the state equation, the transposition
adjoint equation, and the Riesz representation of the corresponding discrete
shape derivative.  It gives a finite-dimensional approximation of the
shape-gradient descent associated with \cref{eq:shape-derivative-main}.
For a shape pair $(D_0,D_1)$, the residual $r$ and its zero extension
$\widetilde r$ are defined by \cref{eq:boundary-residual}.  The adjoint
variable used in the descent method is the transposition solution $(p,q)$
of \cref{eqAdjGIP} with boundary datum $\widetilde r$.
The local regularity result in \cref{thm:adjoint-interior-regularity}
ensures that $p$ and $q$ admit $H^1$-representatives in a neighborhood
of $\partial D_0\cup\partial D_1$.  Hence the traces of $p$ and $q$
and the boundary integrals in \cref{eq:shape-derivative-main} are well
defined.

The finite element and least-squares Monte Carlo discretizations approximate
the continuous state-adjoint optimality system on a common finite ensemble of
Brownian paths.  The same ensemble generates the synthetic boundary
observations.

At the continuous level, a smooth critical shape pair is characterized by
$d\mathcal J(D_0,D_1;V_0,V_1)=0$ for all admissible $(V_0,V_1)$.
For a descent method, the boundary form in
\cref{eq:shape-derivative-main} is represented through a Riesz map on a
Hilbert space of vector fields on $G$.  Set
$\mathcal H:=H_0^1(G;\mathbb{R}^n)$.  Following the linear-elasticity
metrics used in shape optimization
\cite{SchulzSiebenborn2016,EtlingHerzog2020}, the elasticity gradient flow is
induced by the following bilinear form: for fixed Lam\'e parameters $\mu>0$
and $\lambda\geq 0$,\vspace{-2mm}
$$
a_{\rm el}(V,W)
:=
\int_G
\Big(
2\mu \sum_{i,j=1}^n \varepsilon_{ij}(V) \varepsilon_{ij}(W)
+
\lambda\,\operatorname{div}V\,\operatorname{div}W
\Big)\,dx,\vspace{-2mm}
$$
where $\boldsymbol{\varepsilon}(V):=(\nabla V+\nabla V^\top)/2$.

The two independent velocity directions in the Eulerian derivative lead to
the product metric on $\mathcal H\times\mathcal H$.  The
elasticity-smoothed descent direction
$(\vartheta_0,\vartheta_1)\in\mathcal H\times\mathcal H$ is determined by
the variational identity\vspace{-2mm}
\begin{equation}
\label{eqContinuousGradFlow}
a_{\rm el}(\vartheta_0,W_0)+a_{\rm el}(\vartheta_1,W_1)
=
-d\mathcal{J}(D_0,D_1; W_0,W_1),
\quad \forall \, (W_0,W_1) \in \mathcal{H}\times\mathcal H.\vspace{-2mm}
\end{equation}
The component $\vartheta_0$ corresponds to variations of $D_0$ with
$D_1$ fixed, and $\vartheta_1$ corresponds to variations of $D_1$ with
$D_0$ fixed.  The product metric separates only the two components of the
Riesz representation.  However, the Eulerian derivative $d\mathcal J$
still couples the two supports through the state and adjoint variables.

By Korn's inequality, $a_{\rm el}$ is coercive on $\mathcal H$.
Consequently, the product bilinear form on the left-hand side is coercive on
$\mathcal H\times\mathcal H$, so the variational problem uniquely
determines the pair $(\vartheta_0,\vartheta_1)$ associated with this
elasticity metric.  Substitution of
$(W_0,W_1)=(\vartheta_0,\vartheta_1)$ into the variational identity yields\vspace{-2mm}
$$
d\mathcal J(D_0,D_1;\vartheta_0,\vartheta_1)
=
-a_{\rm el}(\vartheta_0,\vartheta_0)
-a_{\rm el}(\vartheta_1,\vartheta_1)
\leq 0.\vspace{-2mm}
$$
For a nonzero solution $(\vartheta_0,\vartheta_1)$, coercivity makes this
quantity strictly negative, so $(\vartheta_0,\vartheta_1)$ is a first-order
descent direction for the objective.

The shape-gradient descent framework is summarized in
Algorithm~\ref{algShapeGradient}, where $\mathcal N$ is the prescribed
iteration count.  Monte Carlo quadrature approximates the expectations in the
misfit and the shape gradient by sample averages over $N_{MC}$ Brownian
paths.  The LSMC approximation of the adjoint equation is based on the same
paths.

\begin{algorithm}
\caption{Shape-gradient descent algorithm for reconstructing $(D_0,D_1)$}
\label{algShapeGradient}
\begin{algorithmic}[1]
\REQUIRE Initial polygonal pair $(D_0^{h,0},D_1^{h,0})$, observation
boundary $\Gamma$, data $g^\delta$, perimeter weights
$\beta_0,\beta_1$, curvature scaling $c_\kappa$, step parameters
$\gamma_0,\eta_\gamma$, and iteration count $\mathcal N$
\FOR{$ k = 0,1,\ldots,\mathcal N-1 $}
\STATE Solve the state equation \cref{eqStateGIP}.
\STATE Solve the adjoint equation \cref{eqAdjGIP}.
\STATE Compute the shape derivative
\cref{eq:shape-derivative-main}.
\STATE Determine the elasticity descent direction
\cref{eqContinuousGradFlow}.
\STATE Update the polygonal interfaces by the deformation map
$\operatorname{Id}+\gamma_k\vartheta_i^h$ and regenerate the
interface-fitted mesh for
$(D_0^{h,k+1},D_1^{h,k+1})$.
\ENDFOR
\end{algorithmic}
\end{algorithm}

\vspace{-3mm}

\subsection{Discretization of the state equation}
\label{subsecStateDiscretization}

The state equation \cref{eqStateGIP} is discretized in space by conforming
$\mathbb P_1$ Lagrange finite elements.  The finite element meshes form a
family of interface-fitted triangular meshes $\{\mathcal T_h\}_{h>0}$.  On
these meshes, the polygonal interfaces $\partial D_0$ and $\partial D_1$ are
represented by mesh edges.
Let $V_h\subset H_0^1(G)$ be the conforming finite element space used for
variational solves with homogeneous Dirichlet condition.  Let
$\widetilde V_h$ be the corresponding continuous piecewise linear space
without the homogeneous boundary condition.  This space provides Dirichlet
lifts.

At each shape iteration the mesh is regenerated from the current polygonal
interfaces.  The discrete supports of the two characteristic sources are
defined by the corresponding element classification.  Elements in
$D_0^h\cap D_1^h$ are included in both supports.  The characteristic
sources are discretized as elementwise constant functions associated with
the polygonal sets $D_i^h$.  Equivalently, their action on test functions
$v_h\in V_h$ is represented through the $L^2(G)$-duality pairing
$(\chi_{D_i^h},v_h)_{L^2(G)}$, $i=0,1$.
For the prescribed target sources used to generate synthetic data, the
characteristic functions are represented by nodal $\mathbb P_1$ interpolants
in $\widetilde V_h$.

The time interval is discretized by a uniform partition
$0=t_0<\cdots<t_{N_t}=T$, with $\Delta t=T/N_t$.  The Brownian increments
are defined by $\Delta W_n:=W(t_{n+1})-W(t_n)$, where
$\Delta W_n\sim\mathcal N(0,\Delta t)$.  The semi-implicit
Euler--Maruyama discretization \cite{Kloeden1992} reads\vspace{-2mm}
\begin{equation}
\label{eqSemiDiscrete}
\left\{\begin{aligned}
&u^{n+1} - \Delta t\,\Delta u^{n+1}
=
u^n + \Delta t\,(\chi_{D_0} + a^n u^n) 
+ \big(b^n u^n + \chi_{D_1}\big)\,\Delta W_n
&&\text{in } G,\\
&u^{n+1}=0
&&\text{on } \partial G.
\end{aligned}\right.\vspace{-2mm}
\end{equation}
Here $u^n\approx u(\cdot,t_n)$.  On each Monte Carlo path, the coefficient
fields are $a^n:=a(t_n,\omega,\cdot)$ and $b^n:=b(t_n,\omega,\cdot)$.  The
lower-order drift coefficient $a^n u^n$, the multiplicative-noise
coefficient $b^n u^n$, and the stochastic source term $\chi_{D_1}$ are
treated explicitly at the previous time level.  The Laplacian is treated
implicitly.  The standard Galerkin projection of
\cref{eqSemiDiscrete} onto $V_h$ defines the fully discrete state equation.
For the Monte Carlo approximation, the discrete state equation is solved
for $N_{MC}$ Brownian paths.  The resulting state trajectories enter both
the boundary residual and the terminal-to-initial adjoint discretization
used below.

\vspace{-2mm}

\subsection{Discretization of the adjoint equation}
\label{subsecAdjointDiscretization}

The adjoint discretization uses the boundary residual associated with the
discrete state.  The boundary flux in the discrete cost is recovered from
the variational time step.  This residual-flux recovery avoids differentiating
the piecewise linear state on boundary facets.  Let $\Gamma_h$ be the subset
of boundary nodes lying on the observation boundary
$\Gamma\subset\partial G$.  For the $m$-th Brownian path, define the nodal
boundary residual at time $t_n$ by\vspace{-2mm}
$$
r_{m,j}^n
:=
\partial_n^h u_m^{n,h}(x_j)-g_{m,j}^{\delta,n},
\qquad x_j\in\Gamma_h,\vspace{-1mm}
$$
and denote its nodal zero extension to $\partial G$ by
$\widetilde r_m^n$.

The adjoint discretization is based on the time grid of
\cref{subsecStateDiscretization} and the terminal condition $p^{N_t}=0$.
The time stepping follows standard BSDE and BSPDE schemes
\cite{Zhang2017,Li2022a,Lue2022}.  At each time level, adaptedness requires
$p^n$ and $q^n$ to be $\mathcal F_{t_n}$-measurable.  The martingale
integrand is determined through conditional expectations together with
$p^n$.  The Euler--Maruyama relation from $t_{n+1}$ to $t_n$ uses
an implicit Laplacian and treats the lower-order terms explicitly:\vspace{-2mm}
\begin{align}
\label{eqBSPDEincrement}
p^{n+1} - p^n
=
-\big(
\Delta p^n
+a^n\bar p^{\,n+1}
+b^n q^n
\big)\Delta t
+ q^n\,\Delta W_n
\quad \text{in } G,
\end{align}
where $\bar p^{\,n+1}:=\mathbb E_n[p^{n+1}]$.  The replacement
$p^n\mapsto \bar p^{\,n+1}$ in the lower-order drift term is first-order
consistent because $p^n=\mathbb E_n[p^{n+1}]+O(\Delta t)$.  This
discretization preserves the structure of the continuous adjoint equation
$dp=-(\Delta p+a p+bq)\,dt+q\,dW(t)$.
The implicit Laplacian guarantees
unconditional stability for the terminal-to-initial time sweep.

The conditional-expectation framework for discrete-time BSDEs
\cite{Bouchard2004,Bouchard2009} is then applied to
\cref{eqBSPDEincrement}.  Multiplication by $\Delta W_n$ followed by
conditioning on $\mathcal F_{t_n}$ extracts the martingale component.
The resulting conditional-expectation system is\vspace{-2mm}
\begin{equation}
\label{eqQcontinuous}
\left\{\begin{aligned}
&q^n
=
\frac{1}{\Delta t}\,
\mathbb E_n\!\big[p^{n+1}\,\Delta W_n\big]
&&\text{in } G,\\
&p^n - \Delta t\,\Delta p^n
=
(1+\Delta t\,a^n)\,\mathbb E_n[p^{n+1}]
+ \Delta t\,b^n q^n
&&\text{in } G,\\
&p^n = \widetilde r^{\,n}
&&\text{on } \partial G.
\end{aligned}\right.\vspace{-2mm}
\end{equation}
The Galerkin formulation of the second equation in \cref{eqQcontinuous} on
$V_h$ defines the fully discrete adjoint equation.  The nonhomogeneous
boundary datum is incorporated by a lift in $\widetilde V_h$.  The elliptic
part has the same bilinear form as the state time step.

\vspace{-2mm}

\subsection{Least-squares Monte Carlo regression}
\label{subsecLSMC}

The conditional expectations
$\mathbb E_n[p^{n+1,h}]$ and
$\mathbb E_n[p^{n+1,h}\Delta W_n]$ in
\cref{eqQcontinuous} are approximated by least-squares projections onto a
finite-dimensional regression space.  For a pathwise
quantity $X^{(m)}$, define\vspace{-2mm}
$$
\widehat{\mathbb E}_n[X]
:=
\operatorname*{argmin}_{\phi\in \operatorname{span}\Phi_n}
\frac{1}{M}\sum_{m=1}^M
\left|X^{(m)}-\phi(\xi_n^{(m)})\right|^2,\vspace{-2mm}
$$
where $\xi_n^{(m)}$ is the vector of regression variables along the
$m$-th Brownian path.  
The regression basis is\vspace{-2mm}
\begin{equation}
\label{eqLSMCbasis}
\Phi_n =
\big(
1,\;
W(t_n),\;
W(t_n)^2,\;
\bar u_G^n,\;
\bar u_{\partial D}^n
\big)^\top
\in \mathbb R^5,\vspace{-2mm}
\end{equation}
where $\bar u_G^n$ is the finite element nodal average of the state over
$G$.  The interface average is\vspace{-2mm}
$$
\bar u_{\partial D}^n
:=
\frac{
\int_{\partial D_0^h} u^{n,h}\,d\sigma_h
+
\int_{\partial D_1^h} u^{n,h}\,d\sigma_h
}{
|\partial D_0^h|_h
+
|\partial D_1^h|_h
}.\vspace{-2mm}
$$
Only nonempty interfaces are included in this average.  The basis
$\Phi_n$ depends on the current polygonal pair.

This regression basis is deliberately low-dimensional. The first three functions capture the leading dependence on the scalar Brownian history, while the two state averages summarize the forward trajectory in the bulk and near the moving interfaces. This choice is a variance-control compromise for the ensemble size used below ($M=N_{\rm MC}=200$), not an optimal one: increasing the basis dimension improves approximation capacity, but also makes the least-squares problems more ill-conditioned and increases sampling error. The numerical examples therefore use this fixed five-dimensional basis throughout.

The scaled variable $p^{n+1,h}\Delta W_n/\Delta t$ is the target of the second projection, so the approximation corresponds to the discrete martingale coefficient $q^n$. The LSMC approximation introduces a projection error from the chosen regression space and a sampling error from the finite ensemble. Under standard assumptions on the regression basis, the statistical error is of order $K/M$ with $K=\dim\operatorname{span}\Phi_n$; the projection error is not estimated separately.

The adjoint time-stepping uses the same $N_{MC}$ Brownian paths as the discrete state equation, with terminal condition $p^{N_t,h,(m)}=0$ for $m=1,\dots,N_{MC}$. These trajectories also generate the synthetic boundary observations; this common-random-numbers coupling reduces the Monte Carlo variance of the boundary residual. At each time level $n=N_t-1,\dots,0$, least-squares Monte Carlo regression with the basis $\Phi_n$ approximates the conditional expectations in \cref{eqQcontinuous}. The resulting approximation of $q^n$ enters the second equation of \cref{eqQcontinuous}, and the finite element solution then determines $p^{n,h}$ on the sample ensemble.
 
\vspace{-2mm}

\subsection{Discrete shape gradient and mesh update}
\label{subsecShapeGradComp}

At each shape iteration the two interfaces are represented by closed polygonal curves $\partial D_i^h$ with outward normals $\nu_i^h$, $i=0,1$. The boundary densities in \cref{eq:shape-derivative-main} are discretized by finite element traces of the adjoint variables and by the first variation of the polygonal perimeter, with effective discrete perimeter weights $\beta_i^h:=c_\kappa\beta_i$. The scaling $c_\kappa$ belongs to the discrete perimeter regularization and is not an additional term in the continuous shape derivative.

The time integral and expectation in the $p$- and $q$-terms are
approximated by time quadrature and Monte Carlo averaging.  The resulting
boundary densities are of the form\vspace{-2mm}
$$
\mathcal G_0^h
=
-\frac{1}{M}\sum_{m=1}^{M}
\sum_{n=0}^{N_t-1}
\Delta t\,
p_{m}^{n,h}\big|_{\partial D_0^h}
+\beta_0^h\kappa_0^h,
$$
and\vspace{-2mm}
$$
\mathcal G_1^h
=
-\frac{1}{M}\sum_{m=1}^{M}
\sum_{n=0}^{N_t-1}
\Delta t\,
q_{m}^{n,h}\big|_{\partial D_1^h}
+\beta_1^h\kappa_1^h,\vspace{-2mm}
$$
where $\kappa_i^h$ is the discrete curvature of
$\partial D_i^h$.  The density $\mathcal G_0^h$ is associated with the
adjoint state component $p$.  The density $\mathcal G_1^h$ is associated
with the martingale component $q$.  With these densities, the discrete
shape derivative is represented in boundary form by\vspace{-2mm}
\begin{equation}
\label{eqDiscreteShapeDerivative}
d\mathcal J_h(D_0^h,D_1^h;V_0^h,V_1^h) =
\int_{\partial D_0^h}
\mathcal G_0^h\, V_0^h\cdot\nu_0^h\,d\sigma_h 
+
\int_{\partial D_1^h}
\mathcal G_1^h\, V_1^h\cdot\nu_1^h\,d\sigma_h.\vspace{-2mm}
\end{equation}
The quantities $\mathcal G_i^h$ are scalar boundary densities of the
discrete Eulerian derivative; they are not themselves interface velocities.
The formal normal field $-\mathcal G_i^h\nu_i^h$ corresponds only to an
unsmoothed boundary $L^2$-descent direction.  The algorithm instead obtains
the interface velocity by applying the negative Riesz map to the boundary
functional.  The finite element counterpart of the product-space elasticity
metric in \cref{eqContinuousGradFlow} maps this boundary functional to descent
fields in $\mathcal H_h$.  The finite-dimensional deformation space is
$\mathcal H_h:=[V_h]^d\subset\mathcal H$.  It is the vector-valued
counterpart of the scalar finite element space used for the state and
adjoint equations.  The positive constants $s_i$ are interpreted as
weights in the product-space Riesz metric:\vspace{-2mm}
$$
a_{{\rm el},s}
\big((V_0,V_1),(W_0,W_1)\big)
:=
s_0^{-1}a_{\rm el}(V_0,W_0)
+
s_1^{-1}a_{\rm el}(V_1,W_1).\vspace{-2mm}
$$
The choice $s_i=1$ gives the unweighted product metric.  Other positive
values modify the metric that represents the same discrete shape derivative
and precondition the descent equation.  With full-boundary observation, the
boundary density is not modified:
$\widetilde{\mathcal G}_i^h=\mathcal G_i^h$.
The discrete descent pair
$(\vartheta_0^h,\vartheta_1^h)\in\mathcal H_h\times\mathcal H_h$ is defined by the following identity: for all $(W_0^h,W_1^h)\in\mathcal H_h\times\mathcal H_h$,\vspace{-2mm}
\begin{equation}
\label{eqDiscreteElasticityGradient}
s_0^{-1}a_{\rm el}(\vartheta_0^h,W_0^h)
+
s_1^{-1}a_{\rm el}(\vartheta_1^h,W_1^h)
=
-\int_{\partial D_0^h}
\widetilde{\mathcal G}_0^h\,
W_0^h\cdot\nu_0^h\,d\sigma_h 
-
\int_{\partial D_1^h}
\widetilde{\mathcal G}_1^h\,
W_1^h\cdot\nu_1^h\,d\sigma_h.\vspace{-2mm}
\end{equation}
With partial-boundary observation, a positive boundary preconditioner is included
in the Riesz representation:
$\widetilde{\mathcal G}_i^h=\omega_i^{(k)}\mathcal G_i^h$.
The factor $\omega_i^{(k)}$ changes only the discrete descent field through
the Riesz map.  It is not part of the continuous Eulerian derivative in
\cref{eq:shape-derivative-main}.
The distance quantities are
$d_{\Gamma}(x):=\operatorname{dist}(x,\Gamma)$ and
$\bar d_i^{(k)}:=\max_{y\in\partial D_i^{h,k}} d_{\Gamma}(y)$.
The boundary preconditioner is defined by\vspace{-2mm}
\begin{equation}
\label{eqPartialObservationWeight}
\omega_i^{(k)}(x):=
\max\left\{\omega_{\min},
\left(\frac{d_{\Gamma}(x)}{\bar d_i^{(k)}}\right)^{p_i}\right\},
\qquad x\in\partial D_i^{h,k}.\vspace{-2mm}
\end{equation}
This factor is a preconditioner for partial boundary observations, rather
than a modification of the underlying shape derivative.  Let $\gamma_k>0$
be the accepted shape-update step length.  The polygonal interfaces are
updated by the deformation map\vspace{-2mm}
\begin{equation}
\label{eqDiscreteInterfaceUpdate}
\partial D_0^{h,k+1}
=
\bigl(\operatorname{Id}+\gamma_k\vartheta_0^h\bigr)
\bigl(\partial D_0^{h,k}\bigr),
\qquad
\partial D_1^{h,k+1}
=
\bigl(\operatorname{Id}+\gamma_k\vartheta_1^h\bigr)
\bigl(\partial D_1^{h,k}\bigr).\vspace{-2mm}
\end{equation}
The sign convention is that of \cref{eqContinuousGradFlow}: the velocity
$\vartheta_i^h$ already contains the negative sign from the Riesz
representation and is therefore the descent field used in the update.
Equivalently, one may solve the positive Riesz problem for
$z_i^h=-\vartheta_i^h$ and update the interface nodes by
$x\leftarrow x-\gamma_k z_i^h(x)$; this is only a change of notation.  Only
geometrically admissible candidate interfaces are retained.

The finite element mesh is then regenerated so that both polygonal interfaces
are represented by mesh edges.  This avoids
relying on a global elastic mesh deformation for two independently moving
internal interfaces, whose close approach or intersection may degrade element
quality.  The resulting remeshing cost is a limitation of the present
discretization, especially for three-dimensional extensions.
The accepted shape-update step length has the form
$\gamma_k=\gamma_0\,\eta_\gamma^k\,2^{-m_k}$,
where $\eta_\gamma\in(0,1]$ is the decay factor and $m_k\ge0$ is selected
by the geometric admissibility condition.

\vspace{-2mm}

\section{Numerical experiments}
\label{secNumericalExperiments}

This section presents numerical reconstructions using the discrete shape-gradient method of \cref{secNumericalAlgorithm}. The examples illustrate how the drift and diffusion source roles in \cref{eq:shape-derivative-main} separate under representative boundary observations.

Let $\Gamma_{\rm bottom}:=\partial G\cap\{x_2=0\}$ be the bottom side
of $G$.  The observation boundaries are
$\Gamma^{\rm full}:=\partial G$ and
$\Gamma^{\rm part}:=\partial G\setminus\Gamma_{\rm bottom}$.
The set $\Gamma^{\rm part}$ consists of the top, left, and right sides of
the unit-square boundary.  In the partial-data cases, the descent equation
includes the boundary preconditioner defined in \cref{eqPartialObservationWeight}.

All reconstructions use $G=(0,1)^2$, initial datum $u_0=0$, mesh parameter
$h\approx 1/50$, final time $T=1$ discretized into $N_t=100$ uniform time
steps, $M=N_{\rm MC}=200$ Monte Carlo samples, and iteration count
$\mathcal N=50$.  The finite element discretization follows
\cref{secNumericalAlgorithm} and is implemented in FEniCSx/DOLFINx
\cite{Baratta2023DOLFINx}.  The coefficients in \cref{eqStateGIP} are
prescribed by
$a(t,\omega,x)=0.1\,\tanh(W(t,\omega))\sin(\pi x_1)$ and $b\equiv 0.5$.
For the synthetic-data experiments, the exact source supports used to generate
the data are denoted by $(D_0^\dagger,D_1^\dagger)$.  For each configuration,
$g_h$ denotes the noise-free discrete Neumann trace of the state associated
with this target pair on the observation grid.  The noisy data used in
\cref{eqDiscreteFunctional} are generated by adding a centered Gaussian
perturbation,\vspace{-2mm}
$$
g_h^{\delta_{\rm nom}}
=
g_h+\delta_{\rm nom}\|g_h\|_{Y_h} Z_h,\vspace{-2mm}
$$
where the entries of $Z_h$ are independent standard normal random variables and
$\|\cdot\|_{Y_h}$ denotes the discrete version of the
$L^2_{\mathbb F}(0,T;L^2(\Gamma))$ observation norm. The entries
$g_{m,j}^{\delta,n}$ in \cref{eqDiscreteFunctional} are the components of
$g_h^{\delta_{\rm nom}}$. 
Unless otherwise
stated, the reconstructions below use $\delta_{\rm nom}=0.01$.
The reconstruction configurations are summarized in
\cref{tabReconstructionDesign}, where $B_\rho(c)$ is the open disk with
center $c$ and radius $\rho$. 

\begin{table}[htbp]
\centering
\scriptsize
\renewcommand{\arraystretch}{1.15}
\setlength{\tabcolsep}{3pt}
\begin{tabular}{@{}llll@{}}
\hline
observation & configuration & target supports & initial geometries \\
\hline
$\Gamma^{\rm full}$ & single drift support
& $D_0^\dagger=B_{0.1}(0.5,0.5)$, $D_1^\dagger=\emptyset$
& $D_0^{h,0}=[0.25,0.75]^2$ \\
$\Gamma^{\rm full}$ & single diffusion support
& $D_0^\dagger=\emptyset$, $D_1^\dagger=B_{0.1}(0.5,0.5)$
& $D_1^{h,0}=[0.25,0.75]^2$ \\
$\Gamma^{\rm part}$ & single drift support
& $D_0^\dagger=B_{0.1}(0.5,0.5)$, $D_1^\dagger=\emptyset$
& $D_0^{h,0}=[0.25,0.75]^2$ \\
$\Gamma^{\rm part}$ & single diffusion support
& $D_0^\dagger=\emptyset$, $D_1^\dagger=B_{0.1}(0.5,0.5)$
& $D_1^{h,0}=[0.25,0.75]^2$ \\
$\Gamma^{\rm full}$ & two overlapping supports
&
\begin{tabular}[t]{@{}l@{}}
$D_0^\dagger=B_{0.1}(0.45,0.5)$,\\
$D_1^\dagger=B_{0.1}(0.55,0.5)$
\end{tabular}
&
\begin{tabular}[t]{@{}l@{}}
$D_0^{h,0}=[0.220,0.700]\times[0.260,0.740]$,\\
$D_1^{h,0}=[0.330,0.810]\times[0.260,0.740]$
\end{tabular} \\
\hline
\end{tabular}
\caption{Reconstruction configurations and initial geometries.}
\label{tabReconstructionDesign}
\end{table}

The optimization parameters were selected empirically from preliminary
computations and then fixed for the reconstructions reported below.  These
choices belong to the discrete descent scheme and are not part of
the continuous shape derivative or the analytical results.
The elasticity Riesz map and shape update use the parameters in
\cref{tabOptimizationParameters}.  The values $s_0$ and $s_1$ are metric
weights in the product-space Riesz representation.  The parameter
$c_\kappa$ defines the effective discrete perimeter weights
$\beta_i^h=c_\kappa\beta_i$.  The boundary preconditioner
$\omega_i^{(k)}$ appears only for partial boundary observations.

\begin{table}[htbp]
\centering
\scriptsize
\renewcommand{\arraystretch}{1.15}
\setlength{\tabcolsep}{4pt}
\begin{tabular}{@{}lllllll@{}}
\hline
configuration
& $\mu,\lambda$
& $\beta_0,\beta_1$
& $\gamma_0$
& $\eta_\gamma$
& $c_\kappa$
& additional parameters \\
\hline
single support, $\Gamma^{\rm full}$
& $1,5$
& $0.005,0.005$
& $2.7$
& $0.95$
& $9.995$
&
\begin{tabular}[t]{@{}l@{}}
$s_0=2.0$ for drift,\\
$s_1=0.84$ for diffusion
\end{tabular} \\
single support, $\Gamma^{\rm part}$
& $1,5$
& $0.005,0.005$
& $2.4$
& $0.953$
& $10$
&
\begin{tabular}[t]{@{}l@{}}
$s_0=2.0$ for drift,\\
$s_1=1.45$ for diffusion,\\
$\omega_{\min}=0.4,\ p_0=0,\ p_1=1.5$
\end{tabular} \\
overlapping supports, $\Gamma^{\rm full}$
& $1,5$
& $0.005,0.005$
& $2.7$
& $0.95$
& $10$
& $s_0=1.2,\ s_1=0.85$ \\
\hline
\end{tabular}
\caption{Optimization and Riesz-metric parameters.}
\label{tabOptimizationParameters}
\end{table}

The reconstructions indicate that the two components of the discrete boundary
derivative separate the two source channels.
\Cref{figFullDiffusionReconstruction} combines the geometric evolution and the
corresponding discrete objective history for the reconstruction of the
diffusion support $D_1^\dagger$ with $\Gamma=\Gamma^{\rm full}$.  In
\cref{figFullDiffusionShapeEvolution},
$\partial D_1^\dagger$ is the target interface used to generate
$g_h^{\delta_{\rm nom}}$, and $\partial D_1^{h,k}$ is the reconstructed polygonal
interface after $k$ iterations.  The polygonal interface moves toward the
prescribed support.  Consistently, \cref{figFullDiffusionCost} shows a rapid
initial decrease of the discrete objective followed by slower refinement.

\begin{figure}[!t]
\centering
\captionsetup[subfigure]{font=small,skip=3pt}
\begin{subfigure}[t]{0.96\textwidth}
\centering
\includegraphics[width=\linewidth]
{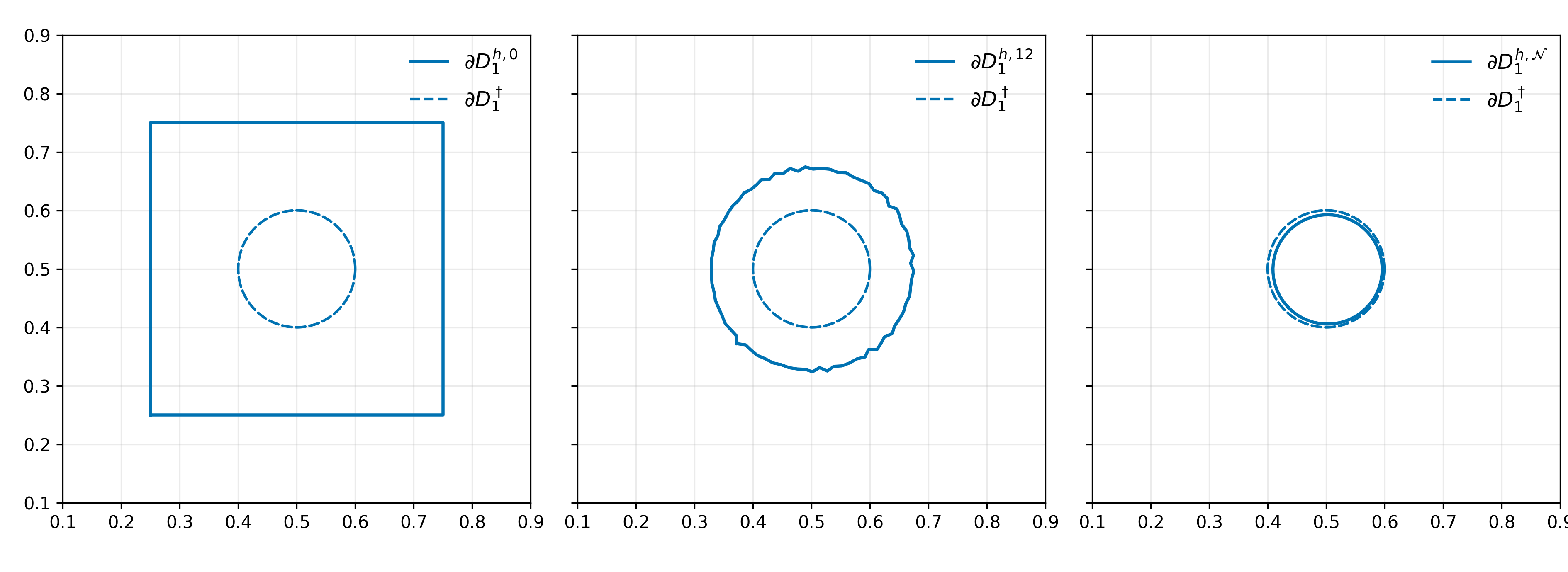}
\caption{Interface evolution for $k=0,12,50$.}
\label{figFullDiffusionShapeEvolution}
\end{subfigure}

\medskip

\begin{subfigure}[t]{0.48\textwidth}
\centering
\includegraphics[width=\linewidth]
{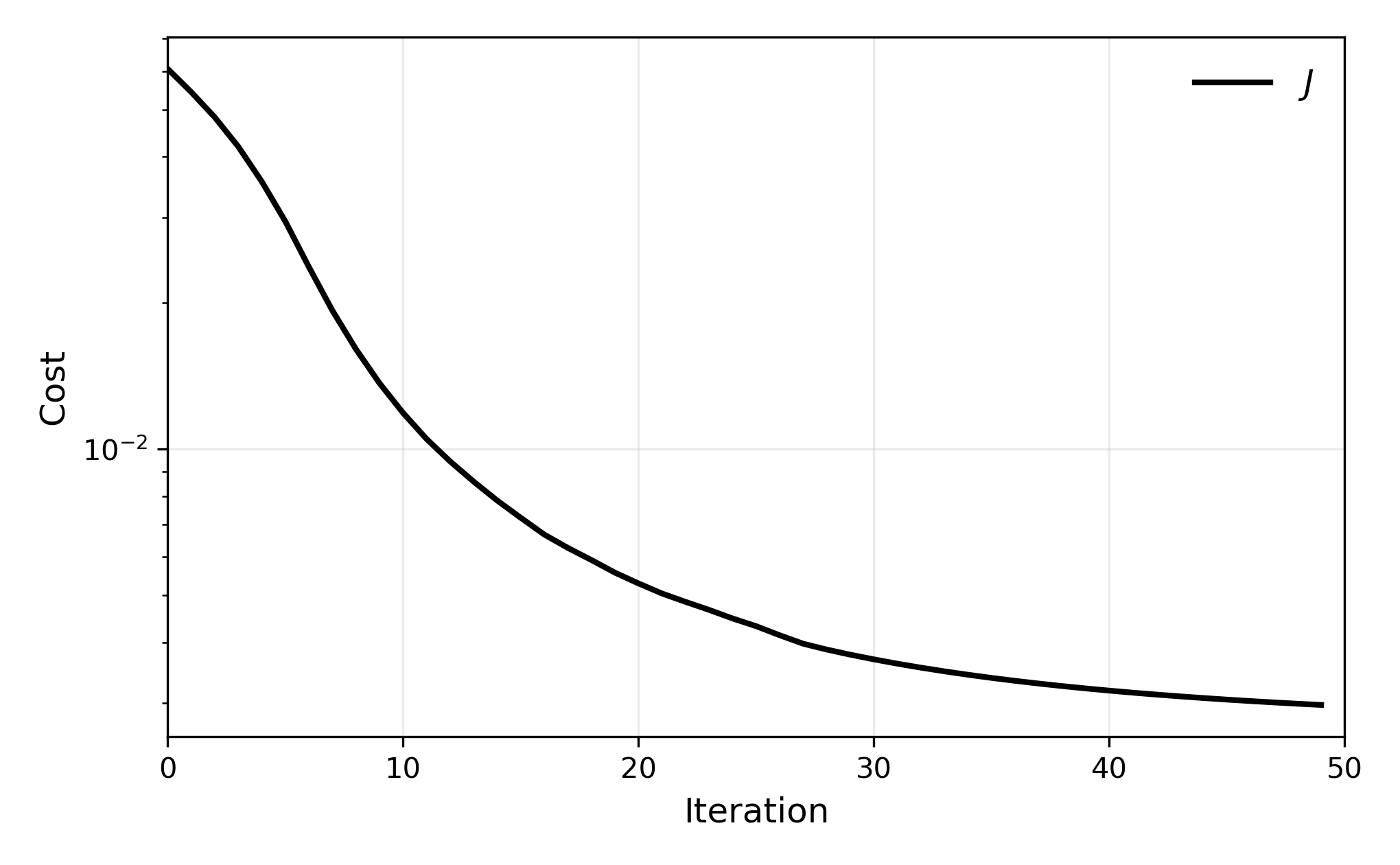}
\caption{Discrete objective history.}
\label{figFullDiffusionCost}
\end{subfigure}
\caption{Full-boundary reconstruction of the diffusion support
$D_1^\dagger$.  The solid curves in (a) show the reconstructed interfaces
$\partial D_1^{h,k}$, and the dashed curve shows the target interface
$\partial D_1^\dagger$.}
\label{figFullDiffusionReconstruction}
\end{figure}

\Cref{figFinalReconstructionComparison} compares the final polygonal
reconstructions by observation regime.  With
$\Gamma=\Gamma^{\rm full}$, the single drift-support case and the joint
reconstruction recover the prescribed geometric scale up to the discretization
scale; in the joint case, the two reconstructed interfaces remain
distinguishable.  This is consistent with the separate boundary terms involving
$p$ and $q$ in \cref{eq:shape-derivative-main}.  With
$\Gamma=\Gamma^{\rm part}$, removing the bottom side mainly affects
localization in the direction associated with the unobserved boundary.  The
final supports nevertheless retain the correct qualitative scale for both the
drift and diffusion source roles.

\begin{figure}[!t]
\centering
\captionsetup[subfigure]{font=small,skip=3pt}
\begin{subfigure}[t]{0.7\textwidth}
\centering
\includegraphics[width=\linewidth]
{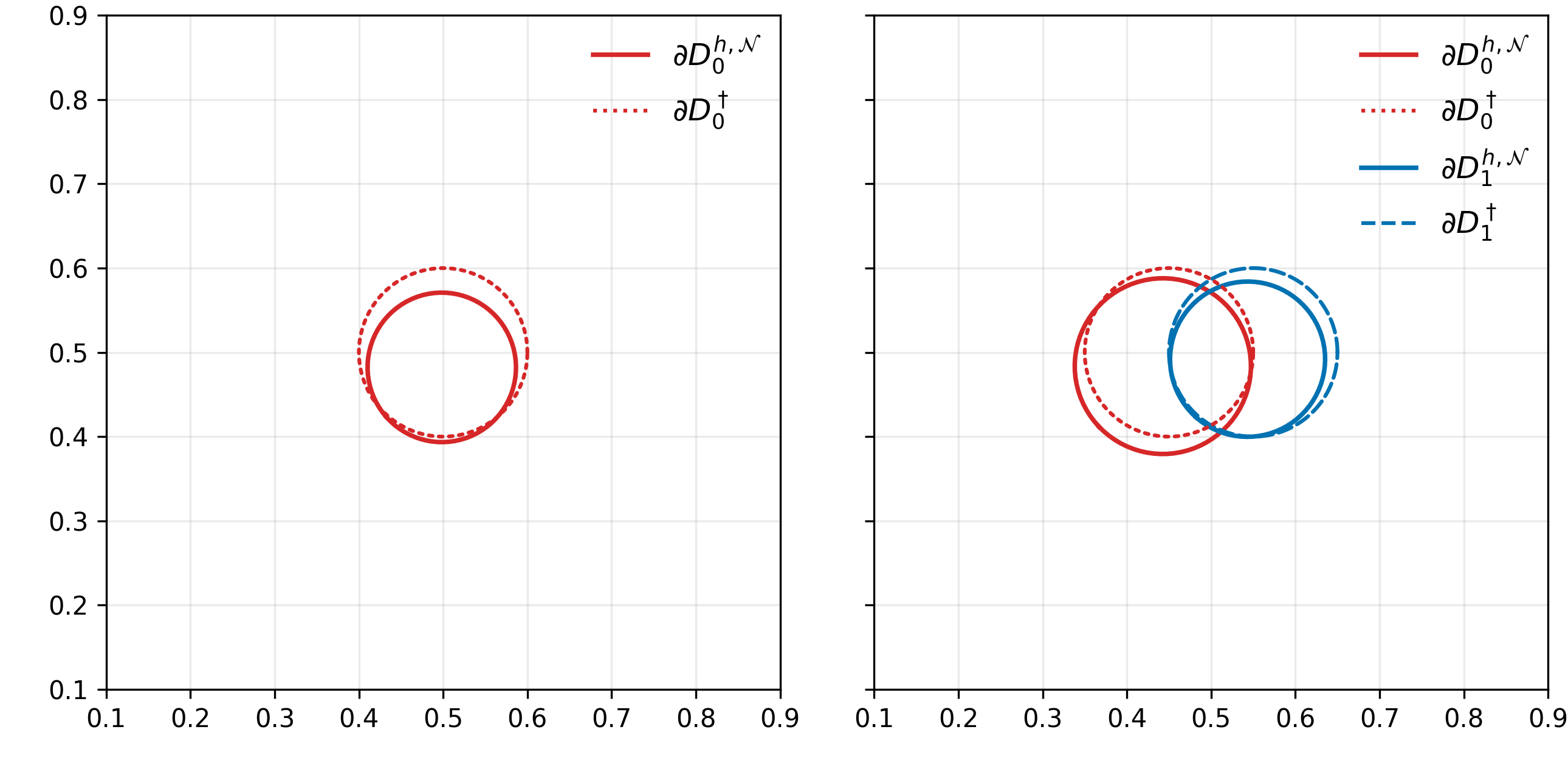}
\caption{Full-boundary observation.}
\label{figFullDriftOverlapOverlay}
\end{subfigure}

\medskip

\begin{subfigure}[t]{0.72\textwidth}
\centering
\includegraphics[width=\linewidth]
{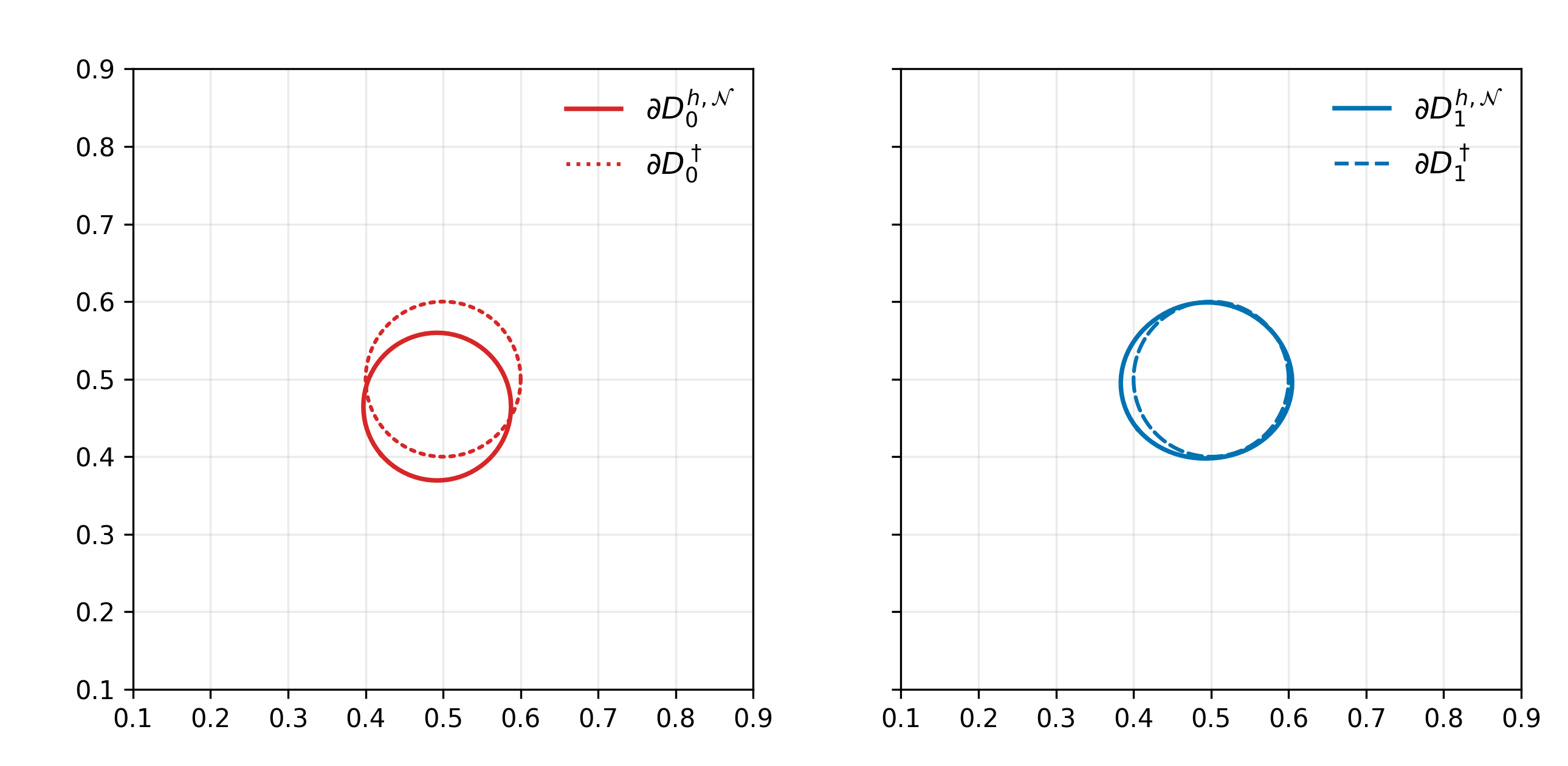}
\caption{Partial-boundary observation with the bottom side excluded.}
\label{figPartialNoBottomOverlay}
\end{subfigure}
\caption{Final polygonal reconstructions under the two observation
regimes.  Panel (a) includes the single drift-support reconstruction and
the joint reconstruction of the target pair
$(D_0^\dagger,D_1^\dagger)$; panel (b) shows the single drift- and
diffusion-support reconstructions from partial boundary data.}
\label{figFinalReconstructionComparison}\vspace{-2mm}
\end{figure}
Together, the interface evolution, objective decay, and final reconstruction comparisons show that the discrete method can separate the two source channels under both full and partial boundary observations.

\vspace{-2mm}

\section*{Acknowledgements}
Yunzhang Li, Meizhi Qian, and Yu Wang sincerely thank Prof.~Enrique Zuazua for his warm hospitality and many fruitful discussions. The authors declare no conflicts of interest.

\vspace{-2mm}

\bibliographystyle{siam-paper-no}
\bibliography{bibitems}

\end{document}